\documentclass[smallextended,referee,envcountsect,]{svjour3}
\smartqed
\usepackage{graphicx}
\journalname{JOTA}

\usepackage{amssymb,amsmath}
\usepackage{hyperref}
\usepackage[capitalise]{cleveref}
\usepackage{algorithm}
\usepackage{algorithmic}
\newcommand{\w}{\mathbf}
\newcommand{\blue}{\textcolor{black}}

\newtheorem{assumption}{Assumption}
\newtheorem{condition}{Condition}
\usepackage{enumitem}
\usepackage{subcaption}
\usepackage{cite}
\usepackage[dvipsnames]{xcolor}

\begin{document}

\title{A Multilevel Method for Self-Concordant Minimization}


\author{Nick Tsipinakis \and Panos Parpas}

\institute{Nick Tsipinakis \at
             UniDistance Suisse \\
              Brig, Switzerland\\
              nikolaos.tsipinakis@unidistance.ch
           \and
              Panos Parpas,  Corresponding author  \at
              Imperial College London \\
              London, U.K.\\
              p.parpas@imperial.ac.uk
}

\date{Received: date / Accepted: date}

\maketitle

\begin{abstract}
The analysis of second-order optimization methods based either on sub-sampling, randomization or sketching has two serious
shortcomings compared to the conventional Newton method. The first shortcoming 
is that the analysis of the iterates has only been shown to be scale-invariant only under specific assumptions on the problem structure. The second shortfall is that the fast convergence rates 
of second-order methods have only been established by making assumptions regarding the input data. 
In this paper, we propose a randomized Newton method for self-concordant functions to address both shortfalls. We propose a Self-concordant Iterative-minimization-Galerkin-based Multilevel
 Algorithm (SIGMA) and establish its super-linear convergence rate using the theory of self-concordant functions.
Our analysis is based on the connections between multigrid optimization methods, and the role of 
coarse-grained or reduced-order models in the computation of search directions. We take advantage of the insights from the
 analysis to significantly improve the performance of second-order
methods in machine learning applications. We report encouraging initial experiments that suggest SIGMA  outperforms
 other state-of-the-art sub-sampled/sketched Newton methods for both medium and large-scale problems. 
\end{abstract}


\section{Introduction}
First-order optimization methods, stochastic, proximal, accelerated or otherwise, are a popular class of algorithms; especially for the
 large-scale optimization models that arise in modern machine learning applications. 
The ease of implementation in distributed architectures and the ability to obtain a reasonably accurate solution quickly are the
 main reasons for the dominance of first-order methods in machine 
learning applications. In the last few years, second-order methods based on variants of the Newton method have also been proposed. 
Second-order methods, such as the Newton method, offer the potential of quadratic convergence rates, and scale invariance. 
Both of these features are highly desirable in optimization algorithms and are not present in first-order methods.

Scale invariance is crucial because it means that the algorithm is not sensitive to the input data (see \cite{pilanci2017newton} for a
 thorough discussion of the consequences of scale invariance in machine learning applications).
Unfortunately, the conventional Newton method has huge storage and computational demands and does not scale to applications that
 have both large and dense Hessian matrices.
To improve the convergence rates, and robustness of the optimization algorithms used in machine learning applications many authors
have recently proposed modifications of the classical Newton method. 
We refer the interested reader to the recent survey in \cite{MR3797719} for a thorough review.
Despite the recent interest, and developments in the application of machine learning applications, existing approaches suffer from
one or both of the following shortfalls that we address in this
paper. 

\textbf{I: Lack of scale-invariant convergence analysis without restrictive assumptions.} The Newton algorithm can be
 analyzed using the elegant theory of self-concordant functions. 
Convergence proofs using the theory of self-concordant functions enable the derivation of convergence rates that are independent
 of the Lipschitz constants and strong convexity parameters of the model\cite{MR1258086}. In machine learning applications these constants are related to the input data of the
 problem 
(e.g., the dictionary in supervised learning applications), so having a theory that is not affected by the scaling of the data is
 quite important both for practical reasons (e.g., choice of step-sizes), and theory (rates derived using this approach do not
  depend 
on these constants). 

\textbf{II: Lack of global convergence guarantees accompanied with a local super-linear convergence rate without ad-hoc
assumptions regarding the spectral properties of the input data.} 
The second major feature of second-order methods is their global theory and their extremely fast and convergence rates near the minimizer. 
Indeed, the analysis of second-order methods should exhibit this feature in order to justify its higher per-iteration and storage costs when compared to first-order methods.

The above shortcomings have significant implications regarding the practical performance of
 second-order methods in machine learning applications 
(see \cite{arjevani2016oracle,berahas2020investigation} for additional discussion). 
To address the two shortcomings we propose a Newton-based multilevel algorithm that can scale to realistic convex models that arise in
large-scale optimization. The method is general in the sense that it does not
assume that the objective function is a sum of functions, and we make no assumptions regarding the data regime. 
Our theoretical analysis is based on the theory of  self-concordant functions and we are able to prove a local super-linear convergence of the algorithm. 
Specifically, the super-linear convergence rate is achieved when search directions are computed in both a deterministic
and randomized manner. For the latter case, we prove convergence in probability. We obtain our theoretical results by drawing parallels between
the second-order methods used in machine learning, and the so-called Galerkin model from the multigrid optimization literature.

\subsection{Related Work}
In this section we discuss the methods most related to our approach and discuss the theoretical and practical limitations of the
current state-of-the-art. We begin with multigrid optimization methods (also known as multilevel methods).  The philosophy behind multigrid methods constitutes the
basis of our approach.

Multilevel optimization algorithms have been shown to be very efficient for large-scale problems and in many cases outperform other state-of-the-art methods, see for instance \cite{nash2000multigrid, gratton2008recursive, MR2587737, ho2019newton}. Nash in
\cite{nash2000multigrid} introduced the MG/Opt method for solving unconstrained convex optimization. 
However, no convergence rates were given in \cite{nash2000multigrid}. Furthermore,  a
“smoothing” step is required when switching to different levels which is inefficient for large-scale applications.  In
\cite{gratton2008recursive}, the authors expand the 
work in \cite{nash2000multigrid} for trust region methods, but the expensive smoothing step is still required. The
work in \cite{MR2587737} proposes a line search method with global linear convergence rate. Additionally, instead of taking the
smoothing step, the authors introduce new conditions which produce effective search directions and yield faster iterations
compared to \cite{nash2000multigrid, MR2587737}. But even the improved conditions in  \cite{MR2587737} can still be inefficient in practice. The most related
multilevel method to ours is proposed in \cite{ho2019newton}. There, the authors show a sub-linear rate that is followed by a local composite (linear-quadratic) convergence rate for strongly convex functions. However, their analysis achieves a slow rate in the first-phase and they do not prove a super-linear rate.

Similar to multilevel methods, subspace Newton methods have also been proposed to overcome the limitations of the standard Newton method 
\cite{qu2016sdna, gower2019rsn, hanzely2020stochastic}. SDNA \cite{qu2016sdna} randomly selects principal sub-matrices of a matrix $\w{M}$ which is, 
by assumption, an upper bound on the Hessian matrix. However, the sketch matrices are not directly applied to the 
Hessian. An extension that overcomes this limitation was proposed for the Randomized Subspace Newton 
method (RSN) \cite{gower2019rsn}, which, as in this work, allows the user to select from a larger set of sketch matrices. Similarly, a stochastic 
subspace variant of the Cubic Newton method was proposed in \cite{hanzely2020stochastic}. When the regularization parameter of the cubic term is set 
to zero the method reduces to our approach and RSN. Nevertheless, the convergence rate for the above method is linear.

Several variants of the Newton method based on sub-sampling or sketching have been proposed
\cite{berahas2020investigation, byrd2011use, byrd2012sample, bollapragada2019exact, NIPS2015_404dcc91, pilanci2017newton,
roosta2019sub, kovalev2019stochastic}.  Sub-sampled Newton methods are not directly comparable with our approach because they are not reduced-order (or subspace) methods. Nevertheless, we discuss them below due to their well established theory and fast convergence rates. The authors in \cite{byrd2011use, byrd2012sample} introduced sub-sampling
techniques on the Hessian matrix of the Newton method, however, they do not provide convergence rates. The work in
\cite{NIPS2015_404dcc91} proves a local composite rate with high probability, however, the objective function must be
written as a sum of $m$ functions and further the method is efficient only for $m \gg n$, where $n$ is the problem dimension. 
The authors in \cite{kovalev2019stochastic} show a local super-linear convergence rate of the sub-sampled Newton
method, and convergence is attained even when the batch size is equal to one, but still the objective function needs be written as a sum of functions.
The method proposed in \cite{bollapragada2019exact} not only sub-samples the Hessian matrix but also the gradient vector. The
method enjoys a local super-linear convergence rate and the results are given in expectation. The authors in 
\cite{roosta2019sub} extended the results in \cite{bollapragada2019exact}, and provide a local
super-linear convergence rate with high probability. 
The method with the most related theoretical results to ours is 
Newton Sketch in \cite{pilanci2017newton}. The authors assume self-concordant functions and thus they are able to provide a global and scale-invariant analysis with a local super-linear convergence rate as we do in this paper. 
However, the square root of the Hessian matrix must be 
available. Furthermore, the method has expensive iterations as
it requires matrix multiplications to form the 
sketched Hessian which is much slower compared to randomly 
selecting data points. To this end, all the above methods fail to address one or both of the shortcomings discussed in the previous section, and when
they do so, further assumptions are required such as regarding the objective function (sum of functions) or knowledge of the square root of the
Hessian matrix. 

\subsection{Contributions}

In view of the related work, in this paper we propose the Self-concordant Iterative-minimization Galerkin-based Multilevel
Algorithm (SIGMA) for unconstrained high-dimensional convex programs, and we provide convergence analysis that addresses both
shortcomings discussed earlier in this section. Below we discuss in detail the main contributions of this paper. We begin with the main theoretical properties of SIGMA.

The proposed approach is based on the multilevel framework of \cite{ho2019newton}, but instead, we assume that the objective
function is strictly convex and self-concordant. Similar to \cite{ho2019newton}, we employ the Galerkin model as the target coarse-grained model. 
We show that the Galerkin model is first- and second-order coherent and we provide connections with the classical Newton method. Such connections are 
important as they enhance our intuition regarding the subspace methods discussed in the previous section. Moreover, in \cite{MR2587737, 
ho2019newton}, the authors make use of conditions that yield effective coarse directions. Instead, in this work we propose alternative conditions 
that are tailored to self-concordant functions. In \cref{thm418}, we show that, when the Galerkin model is constructed in a deterministic manner, 
SIGMA enjoys a globally convergent first phase (damped phase) with linear rate followed by a local composite rate. The theorem shows that the proposed method can achieve a local super-linear or quadratic convergence rate. A complete convergence behavior, that accounts for when the method alternates between SIGMA and the Newton method is presented in \cref{corollary complete convergence}.
In \cref{thm422} and \cref{thm quadratic coarse through fine}, we prove a local quadratic convergence rate in the subspace (lower level). 
A proof of the Armijo-rule criterion is also provided. 
 
We also study the Galerkin model that is generated by the Nystr\"om method (sampling according to some distribution function) and we show that the coarse direction can be seen as a randomized approximation of the Newton direction. In \cref{thm seuperlinear probabilities}, we prove the local super-linear convergence rate of SIGMA with probability $1-\rho$. 
More importantly, the theorem indicates how to construct the random coarse directions (i.e., how to select the sampling distribution) in order to achieve convergence guarantees with high-probability and also probability one. 
In particular, we consider the uniform distribution without replacement for constructing the coarse direction (conventional SIGMA) and we discuss cases for which $\rho$ may be small.
Moreover, we introduce a new \emph{adaptive} sampling strategy in order to ensure convergence with probability one. We also introduce a mixed strategy that interpolates between the uniform and adaptive strategies and offers better results than those of the conventional SIGMA method. The idea of introducing new sampling strategies in place of the uniform sampling scheme has been studied in the context of coordinate descent methods \cite{perekrestenko2017faster, flamary2015importance}. However, to the best of our knowledge, this is the first paper that shows the benefits of alternative sampling schemes on multilevel or subspace methods. Furthermore, constructing coarse directions using the Nystr\"om method, we significantly improve 
the computational complexity of the algorithm, since the expensive matrix multiplications can be abandoned. Under this regime, we identify instances such that the expensive conditions which provide effective coarse direction in each iteration are no longer necessary. We note that this is not possible for other multilevel methods such as \cite{ho2019newton, MR2587737}. 
Lastly, in \cref{thm expects} we provide a simple convergence result in expectation which provides a complete picture of the convergence behavior of SIGMA in this regime.



The method is suitable for large-scale optimization. For instance, consider the standard setting in machine learning applications: $f : \w{x} \rightarrow \w{Ax}$, where $\w{A} \in \mathbb{R}^{m \times N}$ is the dataset matrix and let $m, N \gg 1$. 
If we assume that $n \leq N$ (or even $n \ll N$) are the dimensions of the model in the coarse {level}, then SIGMA requires $\mathcal{O}(mn^2)$ operations to form the Hessian, and, therefore, solving the corresponding system of linear equations to obtain search directions can be carried out efficiently. If the objective function cannot be written as a sum of functions we show that the complexity of iterates of SIGMA using the Nystr\"om method is $\mathcal{O}(n^3 + nN)$ which is a significant improvement compared to the Newton method that requires $\mathcal{O}(N^3 + N^2)$ operations. Note also that the standard sub-sampled Newton methods require approximately $\mathcal{O}(N^3)$ operations for solving the system of linear equations and thus they have the same limitations with the conventional Newton method. Thus, the proposed method significantly improves the applicability of second-order methods in high-dimensional settings. 
We illustrate this with numerical experiments for solving the problem of maximum likelihood estimation in large-scale machine learning. In particular, numerical experiments based on standard benchmark problems and other state-of-the-art second-order methods suggest that the conventional SIGMA compares favorably with 
the state-of-the-art. For specific problem structures, say problems with nearly low-rank Hessian matrices, SIGMA is typically several times faster and more robust. In contrast with the sub-sampled Newton methods that they are suitable only when $m \gg n$, we show that SIGMA is suitable for a wide range of problems and both regimes such as $m<N$ and $m>N$. {We perform comparison between the conventional SIGMA and the one that employs the new sampling strategies. Numerical results show substantial improvements in the convergence rates of SIGMA using the new sampling strategies (up to four times faster). These numerical results are on-par with our improved theoretical results. 
Hence, the alternative sampling strategies can considerably accelerate convergence of the multilevel or subspace methods, such as \cite{ho2019newton, MR2587737, gower2019rsn, hanzely2020stochastic}. Finally, numerical experiments clearly show that SIGMA achieves super-linear and quadratic convergence rates. In this case too, the practice is in agreement with the theory.
}

\subsection{Notation and Preliminaries} \label{sec prel}

Throughout this paper, all vectors are denoted with bold lowercase letters, i.e., $\w{x} \in \mathbb{R}^n$ and all matrices with bold uppercase letters, i.e., $\w{A} \in \mathbb{R}^{m \times n}$. The function $\| \w{x} \|_2 = \langle \w{x}, \w{x} \rangle^{1/2} $ is the $\ell_2$- or Euclidean norm of $\w{x}$. The spectral norm of $\w{A}$ is the norm induced by the $\ell_2$-norm on $\mathbb{R}^n$ and it is defined as $\| \w{A} \|_2 := \underset{\| \w{x} \|_2 = 1}{\max} \|\w{A} \w{x} \|_2$. It can be shown that $\| \w{A} \|_2 = \sigma_1(\w{A})$, where $\sigma_1(\w{A})$ (or simply $\sigma_1$) is the largest singular value of $\w{A}$, see \cite[Section 5.6]{MR2978290}. For two symmetric matrices $\w{A}$ and $\w{B}$, we write $\w{A} \succeq \w{B} $ when $\w{x}^T (\w{A} - \w{B}) \w{x} \geq 0$, for all $\w{x} \in \mathbb{R}^n$, or otherwise when the matrix $\w{A} - \w{B}$ is positive  semi-definite. Below we present the main properties and inequalities for self-concordant functions. For a complete analysis see \cite{MR2142598, MR2061575}. A univariate convex function $\phi: \mathbb{R} \rightarrow \mathbb{R}$ is called self-concordant with constant $M_{\phi} \geq 0$ if and only if,
\begin{equation} \label{selfconcord definition}
 | \phi'''(x) | \leq M_{\phi} \phi''(x)^{3/2}.
\end{equation}
Examples of self-concordant functions include but not limited to the linear, convex quadratic, negative logarithmic and negative log-determinant functions. Based on the above definition, a multivariate convex function $f : \mathbb{R}^n \rightarrow \mathbb{R} $ is called self-concordant if $\phi(t) := f(\w{x} + t\w{u})$ satisfies \cref{selfconcord definition} for all $\w{x} \in \operatorname{dom} f, \w{u} \in \mathbb{R}^n$ and $t \in \mathbb{R}$ such that $\w{x} +t \w{u} \in \operatorname{dom} f $. Furthermore, self-concordance is preserved under composition with any affine function. In addition, for any convex self-concordant function $\tilde{\phi}$ with constant  $M_{\tilde{\phi}} \geq 0$ it can be shown that $\phi(x) := \frac{M_{\tilde{\phi}}^2}{4} \tilde{\phi}(x)$ is self-concordant with constant $M_\phi = 2$. 
Next, given $\w{x} \in \operatorname{dom}f$ and assuming that $\nabla^2 f(\w{x})$ is positive-definite we can define the following norms
\begin{equation} \label{definition norm_x}
 \| \w{u} \|_\w{x} := \langle \nabla^2 f(\w{x}) \w{u}, \w{u}  \rangle^{1/2} \ \ \text{and} \ \ 
 \| \w{v} \|_\w{x}^* := \langle [\nabla^2 f(\w{x})]^{-1} \w{v}, \w{v}  \rangle^{1/2},
\end{equation}
for which it holds that $ | \langle \w{u}, \w{v}  \rangle | \leq \| \w{u} \|_\w{x}^* \| \w{v} \|_\w{x}$. Therefore, the Newton decrement can be written
as 
\begin{equation} \label{newton decrement}
 \lambda_f(\w{x}) := \| \nabla f(\w{x}) \|_\w{x}^* = \|[\nabla^2 f(\w{x})]^{-1/2} \nabla f(\w{x})  \|_2.
\end{equation}
In addition, we take into consideration two auxiliary functions, both introduced in \cite{MR2142598}. Define the functions $\omega$ and 
$\omega_*$ such that
\begin{equation} \label{omegas}
 \omega(x) := x - \log(1+x) \quad \text{and} \quad \omega_*(x) := -x - \log(1-x),
\end{equation}
where $ \operatorname{dom} \omega = \{ x \in \mathbb{R} : x\geq0 \}$ and $\operatorname{dom}\omega_* = \{ x \in \mathbb{R} : 0 \leq x < 1 \}$, 
respectively, and $\log(x)$ denotes the natural logarithm of $x$. Moreover, note that both functions are convex and their range is the set of positive real numbers.  Further, from the definition \cref{selfconcord definition}, for $M_{\phi} = 2$, we have that 
\begin{equation*}
 \left|  \frac{d}{dt} \left( \phi''(t)^{-1/2} \right) \right| \leq 1,
\end{equation*}
from which, after integration, we obtain the following bounds
\begin{equation} \label{relation univariate sandwitz}
  \frac{\phi''(0)}{(1 + t \phi''(0)^{1/2})^2} \leq \phi''(t) \leq \frac{\phi''(0)}{(1 - t \phi''(0)^{1/2})^2}
\end{equation}
where the lower and the uppers bounds hold for $t\geq0$ and $t \in [0, \phi''(0)^{-1/2})$,  with $t \in \operatorname{dom}\phi$, respectively 
(see also \cite{MR2061575}). Consider now self-concordant functions on $\mathbb{R}^n$ and let  
$ S(\w{x}) = \left\lbrace \w{y} \in \mathbb{R}^n : \| \w{y} - \w{x} \|_\w{x} < 1 \right\rbrace$. For any $\w{x} \in \operatorname{dom}f$ and 
$\w{y} \in S(\w{x})$, we have that (see \cite{MR2142598})
\begin{equation} \label{relation sandwitz}
   (1 - \| \w{y} - \w{x} \|_\w{x})^2 \nabla^2 f(\w{x}) \preceq \nabla^2 f(\w{y}) \preceq \frac{1}{(1 - \| \w{y} - \w{x} \|_\w{x})^2} 
   \nabla^2 f(\w{x}).
\end{equation}
Throughout this paper, we refer to notions such as 
super-linear and quadratic convergence rates. We define these term below. 
Denote $\w{x}_k$ the iterate generated by an iterative process at 
the $k^{\text{th}}$ iteration. The sub-optimality gap of 
the Newton method for self-concordant functions satisfies 
the bound 
$f(\w{x}_k) - f(\w{x}^*) \leq \lambda_f(\w{x}_k)^2$ which 
holds for $\lambda_f(\w{x}_k) \leq 0.68$ \cite{MR2061575}, and thus one can estimate the convergence rate in terms of the local norm of the gradient. It is known that the Newton method achieves a local quadratic convergence rate. In this setting, we say that a process converges quadratically if $\lambda_f(\w{x}_{k+1}) \leq R \lambda_f(\w{x}_{k})^2$, for some $R > 0$. Furthermore, for $R_1, R_2 > 0$,
a process achieves composite convergence rate if 
$\lambda_f(\w{x}_{k+1}) \leq R_1 \lambda_f(\w{x}_{k}) + R_2 \lambda_f(\w{x}_{k})^2$. In addition to quadratic and composite 
convergence rates, we say that a process converges with super-linear rate if $\lambda_f(\w{x}_{k+1}) / \lambda_f(\w{x}_{k}) \leq R(k)$, for some $R(k) \downarrow 0$.

\section{Multilevel Models for Unconstrained Optimization} \label{sec background}
In this section we summarize results for multilevel methods in optimization. {Sections \ref{sec problem framework}-\ref{sec galerkin model} constitute background material and can be found also in \cite{ho2019newton, MR2587737}. In \cref{sec nystrom intro}, we describe the Nystr\"om method which will be used to construct the coarse direction, and introduce the new sampling strategies. In \cref{sec fine search}, we introduce the conditions that guarantee effective coarse directions and provide general results for self-concordant functions.}

\subsection{Problem Framework and Settings} \label{sec problem framework}
Let $\w{x}_h^*$ be defined as the solution of the following optimization problem, 
\begin{equation*}
\w{x}_h^* =  \underset{\w{x}_h \in \mathbb{R}^N}{\operatorname{arg \ min}} f_h(\w{x}_h)
\end{equation*}
where $f_h: \mathbb{R}^N \rightarrow \mathbb{R}$ is a continuous, differentiable and strictly convex self-concordant
function. Further, we suppose 
that it is bounded below so that a minimizer $\w{x}_h^*$ exists. Below we state our main assumption formally. 

\begin{assumption} \label{ass self-conc}
The function $f_h$ is strictly convex and self-concordant with constant $M_f = 2$.
\end{assumption}

Since this work constitutes an extension of the results in \cite{ho2019newton} to self-concordant functions, we adopt similar notation. We clarify that the subscript $h$ of $f_h$ denotes the \textit{fine} or original model we wish to minimize. In this paper, unlike the
idea of multigrid methods, where a hierarchy of several discretized problems is constructed, we consider only two levels. 
The model in the lower level (lower dimension) is called the \textit{coarse} model. Thus, the idea is to use information from the
coarse model to solve the fine model. As with $h$, we use the subscript $H$ to refer to the coarse level and $f_H$ to refer to the coarse model. Moreover, the dimensions related to the fine and coarse models are denoted with $N$ and $n$, 
respectively, that is, $\operatorname{dom} f_H = \{ \w{x}_H \in \mathbb{R}^n \}$ and $\operatorname{dom} f_h = \{ \w{x}_h 
\in \mathbb{R}^N \}$, where $n \leq N$. In traditional multigrid methods the coarse model is 
typically derived by varying a discretization parameter. In machine learning applications a coarse model can be derived by varying the number of 
pixels in image processing applications \cite{MR3395393}, or by varying the dictionary size and fidelity in 
statistical pattern recognition (see e.g. \cite{MR3572365} for examples in face recognition, and background extraction from video).

To ``transfer'' information from coarse to fine model and vice versa we define $\w{P}$ and $\w{R}$ to be the prolongation and 
restriction operators, 
respectively. The matrix $\w{P} \in \mathbb{R}^{N \times n}$ defines a mapping from coarse to fine level and matrix 
$\w{R} \in \mathbb{R}^{n \times N}$ from fine to coarse. The following assumption on the aforementioned operators is 
typical for multilevel methods, see for 
instance \cite{ho2019newton, MR2587737}.
\begin{assumption} \label{assumption P}
The restriction and prolongation operators $\w{R}$ and $\w{P}$ are connected via the following relation
\begin{equation*}
 \w{P} = \sigma \w{R}^T, 
\end{equation*}
where $\sigma>0$, and $\w{P}$ has full column rank, i.e., $ \operatorname{rank}(\w{P}) = n.$
\end{assumption}
For simplification purposes and without loss of generality we assume that $\sigma = 1$. Using the above operators we construct the coarse model as follows. First, let $\w{x}_{h,k}$ be the $k^{\text{th}}$ iterate with associated gradient $\nabla f_h (\w{x}_{h,k})$. We initiate the coarse 
model with initial point $\w{x}_{H,0} := \w{R} \w{x}_{h,k}$. Then, the optimization problem, at the coarse level, at iteration $k$, takes the following form
\begin{equation} \label{coarse model definition}
 \underset{\w{x}_H \in \mathbb{R}^n}{\operatorname{min}} \psi_H(\w{x}_H) :=  f_H(\w{x}_H) + \langle \w{u}_H, \w{x}_H - \w{x}_{H,0} \rangle , 
\end{equation}
where $\w{u}_H := \w{R} \nabla f_h(\w{x}_{h,k}) - \nabla f_H (\w{x}_{H,0})$ and $f_H : \mathbb{R}^n \rightarrow \mathbb{R}
$. Note that the above 
objective function is not just $f_H(\w{x}_H)$, but, in order for the coarse model to be \textit{first-order coherent}, the 
quantity 
$\langle \w{u}_H, \w{x}_H - \w{x}_{H,0} \rangle$ is added, which ensures that
\begin{equation*}
 \nabla \psi_H (\w{x}_{H,0}) = \w{R} \nabla f_h (\w{x}_{h,k}).
\end{equation*}
In addition to the first-order coherency condition, we also
assume that the coarse model is  \textit{second-order coherent}, i.e., $ \nabla^2 \psi_H (\w{x}_{H,0})  = \w{R} \nabla^2 f_h (\w{x}_{h,k})  \w{P}$. Later we discuss how the so-called \textit{Galerkin} model satisfies both first- and second-order coherency conditions.

\subsection{The Multilevel Iterative Scheme} \label{sec universal method}

The philosophy behind multilevel algorithms is to make use of the coarse model
\cref{coarse model definition} to provide the search directions. Such a direction is called a \textit{coarse direction}. If the search direction is computed using the original model then it will be called a \textit{fine direction}. 

To obtain the coarse direction we first compute
 \begin{equation} \label{coarse d_H}
  \hat{\w{d}}_{H,k} := \w{x}_H^* - \w{x}_{H,0},
 \end{equation}
where $\w{x}_H^*$ is the minimizer of \cref{coarse model definition}, and then we move to the fine level by applying the prolongation operator
\begin{equation} \label{coarse d_h}
 \hat{\w{d}}_{h,k} := \w{P} (\w{x}_H^* - \w{x}_{H,0}).
\end{equation}
Note that the difference in the subscripts in the above definitions is because $\hat{\w{d}}_{H,k} \in \mathbb{R}^n$ while 
$\hat{\w{d}}_{h,k} \in \mathbb{R}^N$. We also clarify that $\w{d}_{h,k}$, i.e., with the ``hat'' omitted, refers to the fine direction.  The update rule of the multilevel scheme is
\begin{equation} \label{gama scheme}
 \w{x}_{h,k+1} = \w{x}_{h,k} + t_{h,k} \hat{\w{d}}_{h,k},
\end{equation}
where $t_{h,k}>0$ is the step-size parameter.

It has been shown in \cite{MR2587737} that the coarse direction is a descent direction. However, this result does not 
suffice for $\hat{\w{d}}_{h,k}$ to always lead to reduction in the value (of the objective) function. By the first-order coherent condition, it is easy to see that when $\nabla f_h 
(\w{x}_{h,k}) \neq 0$ and $\nabla f_h (\w{x}_{h,k}) \in \operatorname{null}(\w{R})$ (i.e., $\w{R} \nabla f_h (\w{x}_{h,k}) 
= 0$) we have that $\w{x}_H^* - \w{x}_{H,0} = 0$, and thus $\hat{\w{d}}_{h,k} = 0$, which implies no progress for the multilevel scheme, \cref{gama scheme}.
To overcome this issue we may replace the coarse direction $\hat{\w{d}}_{h,k}$ with the fine direction $\w{d}_{h,k}$ when the former is ineffective. 
Examples of fine directions include search directions arising from Newton, quasi-Newton and gradient descent methods. This approach is very common in the multigrid literature for solving PDEs 
\cite{MR2587737, gratton2008recursive, lewis2005model}. The following conditions, proposed in \cite{MR2587737}, determine whether or not the fine direction 
should be employed, i.e., we use
$\w{d}_{h,k}$ when,
\begin{equation} \label{fine step euclidean conditions}
 \left\| \w{R} \nabla f_h (\w{x}_{h,k}) \right\|_2 \leq \mu \left\| \nabla f_h (\w{x}_{h,k}) \right\|_2 \ \ \text{or} \ 
 \ 
 \left\| \w{R} \nabla f_h (\w{x}_{h,k}) \right\|_2 \leq \epsilon
\end{equation}
where $\mu \in(0, \operatorname{min}(1, \| \w{R} \|_2))$. Hence, the above conditions need to be checked at each iteration and prevent the use of the coarse direction when
$\w{R} \nabla f_h (\w{x}_{h,k}) = 0$, while $\nabla f_h (\w{x}_{h,k}) \neq 0$ and $\w{x}_{H,0}$ is 
sufficiently close to the 
solution  $\w{x}_{H}^*$, according to some tolerance $\epsilon \in (0, 1)$. In PDE problems, alternating between the coarse and the fine direction is necessary for obtaining the optimal solution. In \cref{sec nystrom},
we show it is possible to select the prolongation operator such that the fine direction {is never taken yet the algorithm can still achieve a super-linear rate with probability 1}. 

\subsection{Coarse Model and Variable Metric Methods} \label{sec rel metric mthds}
In this section we discuss connections between the multilevel and variable metric methods, see also \cite{ho2019newton}. The descent direction of a standard variable metric method is given by,
\begin{equation}\label{variable metric}
\begin{split}
 \w{d}_{h,k} &= \underset{\w{d} \in \mathbb{R}^N }{\operatorname{arg \ min}} \left\lbrace \frac{1}{2} \| \w{Q}^{1/2}  \w{d}
 \|_2^2 + 
 \langle \nabla f_h (\w{x}_{h,k}), \w{d}  \rangle  \right\rbrace \\
  & = - \w{Q}^{-1} \nabla f_h (\w{x}_{h,k}),
\end{split}
  \end{equation}
where $\w{Q} \in \mathbb{R}^{N\times N}$ is a positive definite matrix. If, for instance, $\w{Q} = \nabla^2 f_h
(\w{x}_{h,k})$ 
we obtain the Newton method. If $\w{Q}$ is chosen as the identity matrix we obtain the steepest descent method. Based on the construction of the coarse model in \cref{coarse d_H}, we define $f_H$ to be a quadratic model as follows
\begin{equation*}
 f_H(\w{x}_H) := \frac{1}{2} \left\| \w{Q}_H^{1/2} (\w{x}_H - \w{x}_{H,0})  \right\|_2^2
\end{equation*}
where $\w{x}_{H,0} = \w{R} \w{x}_{h,k} $ and $\w{Q}_H \in \mathbb{R}^{n \times n}$ is a positive definite matrix. Then, the coarse model 
\cref{coarse model definition} takes the following form
\begin{equation} \label{coarse model with f_H}
\underset{\w{x}_H \in \mathbb{R}^n }{\operatorname{min}} \psi_H (\w{x}_H) = \underset{\w{x}_H \in \mathbb{R}^n
}{\operatorname{min}} \left\lbrace 
\frac{1}{2} \left\| \w{Q}_H^{1/2} (\w{x}_H - \w{x}_{H,0})  \right\|_2^2 + \langle \w{R} \nabla f_h(\w{x}_{h,k}),  \w{x}_H -
\w{x}_{H,0} 
\rangle\right\rbrace.
\end{equation}
By \cref{coarse d_H} and  since \cref{variable metric} has a closed form solution, we obtain 
 \begin{equation*} 
\begin{split}
   \hat{\w{d}}_{H,k} &= \underset{\w{d}_H \in \mathbb{R}^n }{\operatorname{arg \ min}} \left\lbrace \frac{1}{2} \|
   \w{Q}_H^{1/2}  \w{d}_{H} \|_2^2 + 
 \langle \w{R} \nabla f_h (\w{x}_{h,k}), \w{d}_{H}  \rangle  \right\rbrace \\
  & = - \w{Q}_H^{-1} \w{R} \nabla f_h (\w{x}_{h,k}).
  \end{split}
 \end{equation*}
Further, by construction of the coarse direction in \cref{coarse d_h} we obtain the coarse direction 
\begin{equation} \label{coarse direction d_h}
 \hat{\w{d}}_{h,k} = - \w{P}\w{Q}_H^{-1} \w{R} \nabla f_h (\w{x}_{h,k}).
\end{equation}
Note that if we naively set $n=N$ and $\w{P} = \w{I}_{N \times N}$ we obtain exactly 
equation \cref{variable metric}.

\subsection{The Galerkin Model} \label{sec galerkin model}

Here we present the Galerkin model which will be used to provide improved convergence results.
The Galerkin model was introduced in multigrid algorithms for optimization in \cite{MR2587737}  where it was experimentally tested and found to compare favorably with other methods when constructing coarse-grained models. The Galerkin model can be considered as a special case of the coarse model \cref{coarse model with f_H} under a specific choice of the matrix $\w{Q}_H$. In particular, let $\w{Q}_H$ be as follows
\begin{equation} \label{definition Q_H}
 \w{Q}_H(\w{x}_{h,k}) := \w{R} \nabla^2 f_h(\w{x}_{h,k}) \w{P}.
\end{equation}
Below we present the Galerkin model and show links with the (randomized) Newton method (see also \cite{MR2587737, ho2019newton}). We start by showing the positive definiteness of $\w{Q}_H$.
\begin{lemma} \label{proposition Q_H pd}
{Let $f_h : \mathbb{R}^N \rightarrow \mathbb{R}$ satisfy \cref{ass self-conc} and suppose that \cref{assumption P} also holds. Then, the matrix $\w{Q}_H(\w{x}_{h, k})$ is positive definite.} 
\end{lemma}
\begin{proof}
This is a direct result of \cref{assumption P} and $\nabla^2 f_h(\w{x}_{h}) \succ 0$.
\end{proof}
Using the definition \cref{definition Q_H} in the coarse model \cref{coarse model with f_H} one can obtain the Galerkin model
\begin{align} \label{eq galerking model} 
\underset{\w{x}_H \in \mathbb{R}^n }{\operatorname{min}} \psi_H (\w{x}_H) := 
  \frac{1}{2} 
\left\| [\w{Q}_H(\w{x}_{h,k})]^{1/2} (\w{x}_H - \w{x}_{H,0})  \right\|_2^2 +  \langle \w{R} \nabla f_h(\w{x}_{h,k}),  \w{x}_H - \w{x}_{H,0} 
\rangle,
\end{align}
where, by \cref{proposition Q_H pd}, the Galerkin model satisfies \cref{ass self-conc}, and since $\w{Q}_H(\w{x}_{h,k})$ is invertible, \cref{eq galerking model} has a closed form solution. That is, we can derive $\hat{\w{d}}_{H,k}$ and $\hat{\w{d}}_{h,k}$ as is shown below
\begin{equation} \label{galerkin d_H}
 \hat{\w{d}}_{H,k} = - [\w{R} \nabla^2 f_h(\w{x}_{h,k}) \w{P}]^{-1} \w{R} \nabla f_h (\w{x}_{h,k}),
\end{equation}
and then we prolongate $\hat{\w{d}}_{H,k}$ to obtain the coarse direction
\begin{equation} \label{galerkin d_h}
 \hat{\w{d}}_{h,k} := - \w{P} \hat{\w{d}}_{H,k} 
 = - \w{P} [\w{R} \nabla^2 f_h(\w{x}_{h,k}) \w{P}]^{-1} \w{R} \nabla f_h (\w{x}_{h,k}).
\end{equation}
Observe also that \cref{galerkin d_H} is equivalent to solving the following linear system
\begin{equation} \label{d_h linear system}
 \w{Q}_H(\w{x}_{h,k}) \w{d}_{H} = - \w{R} \nabla f_h (\w{x}_{h,k}),
\end{equation}
which, by the positive-definiteness of $\w{Q}_H(\w{x}_{h,k})$, has a unique solution. At first glance it may seem that computing $\hat{\w{d}}_{h,k}$ may require the computation of the full Hessian matrix. However, we discuss in \cref{remark Q_H nystrom} how the computation of $\w{Q}_H(\w{x}_{h,k})$ can be done in $\mathcal{O}(N)$ operations (in parallel), and thus computing $\hat{\w{d}}_{h,k}$ requires $\mathcal{O}(n^3)$ operations. Using the Galerkin model in \cref{eq galerking model} one ensures that
\begin{equation*}
    \nabla^2 \psi_H (\w{x}_{H,0})  = \w{R} \nabla^2 f_h (\w{x}_{h,k})  \w{P},
\end{equation*}
i.e., the second-order coherency condition is satisfied. 
Similar to \cref{coarse direction d_h}, for $\w{P} = \w{I}_{N \times N}$ we obtain the Newton direction. If $\w{P}$ is a random matrix we obtain a randomized Newton method. 

\subsection{The Nystr\"om method} \label{sec nystrom intro}

In this section, we discuss the Nystr\"om method for the low-rank approximation of a positive-definite matrix. 
We show connections with the Galerkin model, and how to construct the prolongation and restriction operators. The Nystr\"om method builds a rank-$n$ approximation of a positive definite matrix $\w{A} \in \mathbb{R}^{N \times N}$ as follows (see \cite{drineas2005nystrom} for an introduction and \cite{ho2019newton} for a similar analysis)  
\begin{equation} \label{eq: nystrom approx}
    \w{A} \approx \w{A}_n = \w{A} \w{Y} \left( \w{Y}^T \w{A} \w{Y} \right)^{-1} \w{Y}^T \w{A},
\end{equation}
where $\w{Y} \in \mathbb{R}^{N \times n}$, $\operatorname{rank}(\w{Y}) = n < N$. To see the connection between the Galerkin model and the Nystr\"om method set $\w{A} = \nabla^2 f(\w{x}_{h,k})$, $\w{Y} = \w{P}$, in \cref{eq: nystrom approx} and multiply left and right with $ \left[ \nabla^2 f(\w{x}_{h,k}) \right]^{-1}$,  respectively. Then,
\begin{equation*}
    \left[ \nabla^2 f(\w{x}_{h,k}) \right]^{-1} \approx \w{P} \left( \w{R} \nabla^2 f(\w{x}_{h,k}) \w{P} \right)^{-1} \w{R}.
\end{equation*}
Thus, when $\w{P} \left( \w{R} \nabla^2 f(\w{x}_{h,k}) \w{P} \right)^{-1} \w{R}$ is a good approximation of $\left[ \nabla^2 f(\w{x}_{h,k}) \right]^{-1}$ then we expect $\w{d}_{h,k} \approx \w{\hat{d}}_{h,k}$, i.e., the coarse direction based on the Galerkin model that is generated through the Nystr\"om  method is a good approximation of the Newton direction. For random sampling techniques dedicated on the choice of $\w{P}$, see \cite{gittens2011spectral, NIPS2000_19de10ad, smola2000sparse}. {In this work we are interested in cases were samples are drawn according to some discrete probability distribution $\w{p}$. Formally, we construct the prolongation and restriction operators as described below.} 
\begin{definition} \label{def P}
Let $S_N = \left\lbrace 1,2, \ldots, N \right\rbrace $ and denote $S_n \subset S_N$, with the property that the $n < N$ elements are {randomly}
selected from $S_N$ without replacement {according to $\w{p}$}. Further, assume that $s_i$ is the $i^{\text{th}}$ element of $S_n$. Then the prolongation operator 
$\w{P}$ is generated as follows: The $i^{\text{th}}$ column of $\w{P}$ is the $s_i$ column of $\w{I}_{N \times N}$ and, further, it holds that 
$\w{R} = \w{P}^T$.
\end{definition}

{The above definition of $\w{P}$ and $\w{R}$ clearly satisfies \cref{assumption P}. If we consider iteration-dependent operators, $\w{R}_k$, then, effectively, the above definition indicates that the scheme in \eqref{gama scheme} will update $n$ from $N$ entries of $\w{x}_{h,k}$ at each iteration. The selection will be decided according to the assigned probability distribution. The simplest way to construct $\w{P}$ is by using uniform sampling. In this case all entries of $\w{x}_{h,k}$ have the same probability to be selected. However, although the uniform distribution is preferred in most randomized methods (see for instance \cite{ho2019newton, gower2019rsn, hanzely2020stochastic, NIPS2015_404dcc91, byrd2011use}), in this work we consider alternative sampling techniques. Let $p_{i, k}, i = 1, \ldots, N$ be the probability of selecting the $s_i$ column of $\w{I}_{N \times N}$, such that $\sum_{i=1}^{N} p_i = 1$. Below we specify the three different sampling strategies that we consider in this paper.
\begin{itemize}
    \item \emph{Uniform} distribution: $p_{i, k} = \frac{1}{N}$,  which ensures that all entries of $\w{x}_{h,k}$ will be updated equally. In this case the Nystr\"om method is commonly referred to as \emph{naive}.
    \item \emph{Adaptive} distribution: $p_{i, k} = \frac{|g_i|}{\sum_{i=1}^N |g_i|}$, 
    where $g_i = \frac{\partial{f(\w{x}_{h,k})}}{\partial{x_i}}$. Using the adaptive distribution we assign larger probability to coordinates whose partial derivatives are large. Thus, the entries of $\w{x}_{h,k}$ for which $g_i = 0$ will never be selected.
    \item \emph{Mixed} distribution: $p_{i, k} = (1-\tau) \frac{1}{N} +  \tau\frac{|g_i|}{\sum_{i=1}^N |g_i|}$, where $\tau \in (0,1)$. The mixed strategy effectively interpolates between the uniform and adaptive schemes. 
\end{itemize}
The idea of adaptive sampling strategies has been proposed previously to show significant improvements in convergence rates of coordinate descent methods \cite{perekrestenko2017faster, flamary2015importance}. Here, we will use the adaptive and mixed strategies to significantly improve the theoretical and numerical results of multilevel or subspace methods that rely on a uniform sampling strategy.}

Moreover, 
when $\w{P}$ is chosen randomly then the direction $\hat{\w{d}}_{h,k}$ computed in \cref{galerkin d_h} is also random. Thus, 
the randomness of $\w{P}$ will imply randomness in the step-size $t_{h,k}$ and the iterates $\w{x}_{h,k}$ \cref{gama scheme}. If at iteration k the 
prolongation matrix $\w{P}_k$ is used, then the algorithm generates a sequence of random variables 
$M_k(\omega) = (\w{x}_{h,k}(\omega), t_{h,k}(\omega),   \hat{\w{d}}_{h,k}(\omega), \w{P}_k(\omega))$. We use $\mathcal{F}_k := 
\sigma(M_0, M_1, \ldots, M_k)$ to denote the $\sigma$-algebra generated by \cref{gama scheme} up to iteration k. Below we describe how to efficiently compute the reduced Hessian matrix (\cref{definition Q_H}) given the above definition.

\begin{remark} \label{remark Q_H nystrom}
We note that it is expensive to first compute the Hessian matrix and then form the reduced Hessian matrix $\w{Q}_H(\w{x}_{h,k})$ as it requires 
$\mathcal{O}(N^2)$ operations. Instead, using the definition of $\w{P}$, one may first compute the product $\nabla f(\w{x})^T \w{P} $ 
by sampling $n$ from $N$ entries of the gradient vector. Then, for $\w{P} = [\w{p}_1 \ \w{p}_2 \cdots \w{p}_n]$, computing the gradient of 
$\nabla f(\w{x})^T \w{p}_i, i = 1, \ldots, n $ requires $\mathcal{O}(N)$ operations and thus, in total, $\mathcal{O}(nN)$ operations are required to 
compute the product $ \tilde{\w{P}} := \nabla^2 f(\w{x}) \w{P} $.  We then form the $n \times n$ reduced Hessian matrix in equation 
\cref{definition Q_H} by sampling rows of the matrix $\tilde{\w{P}}$ according to \cref{def P}. Therefore, the total per-iteration cost to form 
the reduced Hessian matrix will be $\mathcal{O}(nN)$. In addition, solving the linear system of equation in \cref{d_h linear system} to obtain the coarse direction requires $\mathcal{O}(n^3)$ operations. Further, note that, for $ i = 1, \ldots, n $ the computation of $ \nabla^2 f(\w{x}) \w{p}_i$ can be carried out independently resulting in $\mathcal{O}(N)$ operations required for computing the reduced Hessian matrix which is the same as the complexity of forming the gradient. As a result, the total per-iteration cost of the proposed method using the naive Nystr\"om method will be $\mathcal{O}(n^3 + (1+n)N) \approx  \mathcal{O}(n^3 +nN)$ or $\mathcal{O}(n^3 + 2N)$ when performing the computations in parallel.
{ The technique described above assumes the use of an automatic differentiation routine to compute the Hessian-vector products. It is explained in detail in \cite{christianson1992automatic}.}
\end{remark}

Building a rank-$n$ Hessian matrix approximation using the Nystr\"om method provides an inexpensive way for constructing the coarse direction in \cref{galerkin d_h}, as it performs sampling without replacement, which significantly reduces the total computational complexity (recall that the total per-iteration cost of the Newton method is $\mathcal{O}(N^2 + N^3)$). 
In addition, in terms of complexity, the proposed method has two main advantages compared to sub-sampled Newton methods. Firstly, sub-sampled Newton methods require approximately $\mathcal{O}(N^3)$ operations for computing the search direction which means that, when $N$ is large, {sub-sampled methods have} the same limitations as those of the conventional Newton method. 
Secondly, they offer improved complexity of iterates when the objective function is written as a sum  of functions. On the other hand, SIGMA is able to overcome the limitations associated with the Hessian matrix since the computations are performed in the coarse level, and, in addition, it offers improved complexity of iterates without requiring the objective to be written as a sum of functions. The aforementioned advantages significantly increase the applicability of multilevel methods in comparison to the sub-sampled or sketch Newton methods in machine learning and large-scale optimization problems. 

The convergence analysis with $\w{P}$ constructed as in \cref{def P} and the three sampling regimes is given in \cref{sec nystrom}. The above definition of the prolongation operator will be used for the numerical experiments in \cref{sec experiments}. 

\subsection{Fine Search Direction} \label{sec fine search}

As discussed in previous section, condition \cref{fine step euclidean conditions} guarantees the progress of the multilevel method by using the 
fine direction $\w{d}_{h,k}$ in place of the coarse direction $\hat{\w{d}}_{h,k}$ when the latter appears to be ineffective. Since in this work we 
consider self-concordant functions we propose alternative conditions to \cref{fine step euclidean conditions}. In particular, we replace the standard Euclidean norm with the norms defined by the matrices $\w{Q}_H(\w{x}_{h,k})$ and $\nabla^2 f_h(\w{x}_{h,k})$. We begin by defining the approximate decrement, a quantity analogous to the Newton decrement in 
\cref{newton decrement}
\begin{equation} \label{gama decrement}
\begin{split}
 \hat{\lambda}_{f_{h}}(\w{x}_{h,k}) & := \left[ (\w{R} \nabla f_h (\w{x}_{h,k}))^T [\w{Q}_H(\w{x}_{h,k})]^{-1} \w{R} \nabla 
 f_h (\w{x}_{h,k}) \right]^{1/2}.\\
 \end{split}
\end{equation}
We clarify that for the rest of this paper, unless specified differently, we denote the fine direction be the Newton direction, 
$\w{d}_{h,k} = - [\nabla^2 f_h(\w{x}_{h,k})]^{-1}  \nabla f_h (\w{x}_{h,k}) $, and, in addition, for simplification, we omit the subscript $f_h$ from 
both approximate and Newton decrements. Note also that the approximate and Newton decrements can be rewritten as
\begin{equation} \label{definition lamdahat and lamda}
 \hat{\lambda}(\w{x}_{h,k}) := \left\| \w{R} \nabla f_h (\w{x}_{h,k})  \right\|_{[\w{Q}_H(\w{x}_{h,k})]^{-1}} \ \ \& \ \
 \lambda(\w{x}_{h,k}) := \left\| \nabla f_h (\w{x}_{h,k}) \right\|_{[\nabla^2 f_h(\w{x}_{h,k})]^{-1}},
\end{equation}
respectively, where, by positive-definiteness of $\w{Q}_H(\w{x}_{h,k})$ and $\nabla^2 f_h(\w{x}_{h,k})$, both norms are well-defined and they 
serve the same purpose as with $\left\| \w{R} \nabla f_h (\w{x}_{h,k})  \right\|_2$
and $\left\| \nabla f_h (\w{x}_{h,k})  \right\|_2$, respectively. The new conditions are presented below as our main assumption and they are useful 
when minimizing self-concordant functions.
\begin{condition} \label{assumption lamda}
 {Let $\mu \in (0,1)$ and $\nu \in (0,1)$. The iterative scheme \cref{gama scheme} employs the coarse direction \cref{galerkin d_h}
if
\begin{equation*} 
 \hat{\lambda}(\w{x}_{h,k}) > \mu \lambda(\w{x}_{h,k}) \quad \ \ \text{and} \quad \ \ \hat{\lambda}(\w{x}_{h,k}) > \nu.
\end{equation*}
}
\end{condition} 

The above condition is analogous to the original conditions \cref{fine step euclidean conditions} and thus it prevents the use of the coarse direction when $\hat{\lambda}(\w{x}_{h,k}) = 0$ while $\lambda(\w{x}_{h,k}) \neq 0$ and also when $\w{x}_{H,0} = \w{x}_{H}^*$. In multilevel methods $\mu$ is a user-defined parameter that determines whether the algorithm selects the coarse or the fine step. The following lemma gives insights on how to select $\mu$ such that the coarse direction is always performed as long as $\hat{\lambda}(\w{x}_{h,k})$ is positive.

\begin{lemma} \label{prop select mu}
Suppose $\hat{\lambda}(\w{x}_{h,k}) > \nu $ for some $\nu \in (0,1)$. Then, for any $\mu \in (0, \min \{ 1, \frac{\nu}{\lambda(\w{x}_{h,0})} \} )$ we have that
\begin{equation*}
    \hat{\lambda}(\w{x}_{h,k}) > \mu \lambda(\w{x}_{h,k}),
\end{equation*}
for any $k \in \mathbb{N}$.
\end{lemma}

The above result verifies our intuition that coarse steps are more likely to be taken when $\mu $ is selected sufficiently small. Note that $\mu < \frac{\nu}{\lambda(\w{x}_{h,0})}$ is a sufficient condition such that the coarse step is always taken. Nevertheless it does not identify all the values of $\mu$ such that \cref{assumption lamda} holds. There might exist $\mu \in (\frac{\nu}{\lambda(\w{x}_{h,0})}, 1) $ such that $ \hat{\lambda}(\w{x}_{h,k}) > \mu \lambda(\w{x}_{h,k})$ remains true.
In addition, as a corollary of \cref{prop select mu}, one can select some $r \in \mathbb{N}$ and $ \mu < \frac{\nu}{\lambda(\w{x}_{h,r})}$ which ensures that the coarse step will be always taken for all $k \geq r$.

\begin{lemma} \label{lemma upper bound lamda_hat}
For any $k \in \mathbb{N}$ the approximate decrement  in \cref{gama decrement} is bounded as follows 
\begin{equation*} \label{ineq bounds lambda hat}
   \hat{\lambda}(\w{x}_{h,k}) \leq \lambda(\w{x}_{h,k}),
\end{equation*}
where $\lambda(\w{x}_{h,k})$ is the Newton decrement in \cref{newton decrement}.
\end{lemma}

The above result shows that $\hat{\lambda}(\w{x}_{h,k})$ can be as much as $\lambda(\w{x}_{h,k})$. Therefore if the user-defined parameter is selected larger than one then SIGMA will perform fine steps only. 

\section{SIGMA: Convergence Analysis} \label{sec analysis}

In this section we provide convergence analysis of SIGMA for strictly convex self-concordant functions. {Unless mentioned otherwise, throughout this section we assume that \cref{ass self-conc} and \cref{assumption P} always hold, and the method alternates between coarse and fine steps according to \cref{assumption lamda}.} Furthermore, we provide two theoretical results that hold \textbf{(i)} for any $\w{P}$ that satisfies \cref{assumption P} (see \cref{sec gen superlinear rate}), and \textbf{(ii)} when $\w{P}$ is selected randomly at each iteration as in \cref{def P}. {In the latter scenario, we show probabilistic convergence results as well as convergence in expectation (see \cref{sec nystrom})}. In both cases, we prove that SIGMA achieves a local super-linear convergence rate. The idea of the proof is similar to that of the classical Newton method where convergence is split into two phases. The full algorithm including a step-size strategy is specified in \cref{gama}. Technical proofs are relegated to the appendix.

\begin{remark} \label{remark gama}
We make an important remark regarding the practical implementation of \cref{gama}. Notice that checking the condition $\hat{\lambda}(\w{x}_{h,k}) > \mu \lambda(\w{x}_{h,k})$ is inefficient to perform at each iteration as it computes the expensive Newton decrement. One can compute the gradients using the Euclidean norms instead \cref{fine step euclidean conditions} as they serve the same purpose and are cheap to compute. Furthermore, \cref{prop select mu} suggests that when  $\mu$ is selected sufficiently small it is likely that only coarse steps will be taken. However, note that we make no assumptions about the coarse model beyond what has already been discussed until now. For example, the fine model dimension $N$ could be very large, while $n$ could be just a single dimension. In such an extreme case, performing only coarse directions will yield a slow progress of the multilevel algorithm and hence, to avoid slow convergence, larger value of $\mu$ will be required. In \cref{sec nystrom} we describe that, given $\w{P}_k$ is as in \cref{def P} and specific problem structures, SIGMA can reach solutions with high accuracy without ever using fine correction steps. Finally, for obtaining our theoretical results we assume the Newton direction as the fine direction, nevertheless, in practice, $\w{d}_{h,k}$ can be a direction arising from variable metric methods (see \cref{sec rel metric mthds}).
\end{remark}

\subsection{Globally Convergent First-Phase}

\begin{algorithm}[t]
\caption{SIGMA}
\label{gama}
\begin{algorithmic}[1]
\STATE Input: $\mu \in (0, 1) $ , $ \ \alpha \in (0,0.5), \ \beta \in (0,1), 
\ \nu \in (0,0.68^2), \ \epsilon \in (\nu,0.68^2), \ \w{P}_k \in \mathbb{R}^{n \times N}$, $\w{x}_{h,0} \in \mathbb{R}^N$  \\ 
\STATE Compute direction \\
\begin{align*}
 \w{d}_k & := 
\begin{cases}
\w{\hat{d}}_{h,k} \ \ \text{from} \ \  (\ref{galerkin d_h}) & \text{if} \ \ \hat{\lambda}(\w{x}_{h,k}) > \mu \lambda(\w{x}_{h,k}) \ \ 
\text{and} \ \  \hat{\lambda}(\w{x}_{h,k}) > \nu \\
\w{d}_{h,k} \ \ \text{from} \ \  (\ref{variable metric}) & \text{otherwise},
\end{cases}
\end{align*}
\STATE Quit if $ - \langle \nabla f_{h,k}(\w{x}_{h,k}), \w{d}_k \rangle  \leq \epsilon \ \  $
\STATE Armijo search: while $f_h(\w{x}_{h,k} + t_k \w{d}_k) > f_h(\w{x}_{h,k}) + \alpha t_{h,k} \nabla f^T_{h,k}(\w{x}_{h,k}) \w{d}_k, \quad t_{h,k} \leftarrow \beta t_{h,k}$
\STATE Update: $  \w{x}_{h,k+1} := \w{x}_{h,k} + t_{h,k} \w{d}_k $, go to 2
\STATE Return $\w{x}_{h,k}$
\end{algorithmic}
\end{algorithm}

We begin by showing reduction in the value of the objective function of \cref{gama} when the backtracking line search (Armijo-rule) is satisfied. We emphasize that this result is global. The idea of the proofs in the following lemmas are parallel with those in \cite{MR2061575, MR2142598}.

\begin{lemma} \label{lemma411}
For any $ \eta> 0$ there exists $\gamma > 0$ such that for any 
$k \in \mathbb{N}$ with $\hat{\lambda}(\w{x}_{h,k}) > \eta \ $ the coarse direction $\w{\hat{d}}_{h,k}$ will yield the following reduction in the value of the objective function
\begin{equation*}
 f_h(\w{x}_{h,k} + t_{h,k} \w{\hat{d}}_{h,k}) - f_h(\w{x}_{h,k}) \leq -\gamma.
\end{equation*}
\end{lemma}

We proceed by estimating the sub-optimality gap. In particular, we  show it can be bounded in terms of the approximate decrement.

\begin{lemma} \label{lemma412}
Let $\hat{\lambda}(\w{x}_{h,k}) < 1$. Then, 
\begin{equation*}
\omega(\hat{\lambda}(\w{x}_{h,k})) \leq f_h(\w{x}_{h,k}) - f_h(\w{x}_h^*) \leq \omega_*(\hat{\lambda}(\w{x}_{h,k})),
\end{equation*}
where the mappings $\omega$ and $\omega_*$ are defined as in \cref{omegas}. 
\end{lemma}

The above result is similar to \cite[Theorem 4.1.11]{MR2142598} but with $ \hat{\lambda}(\w{x}_{h,k})$ in place of $\lambda(\w{x}_{h,k})$. Alternatively, similar to the analysis in \cite{MR2061575}, the sub-optimality gap can be given as follows.

\begin{lemma} \label{lemma suboptimality gap}
If $\lambda(\w{x}_{h,k}) \leq 0.68$, then
\begin{align*}
 f_h(\w{x}_{h,k}) - f_h(\w{x}_h^*)  \leq \lambda(\w{x}_{h,k})^2. 
\end{align*}
\end{lemma}

As a result, $\lambda(\w{x}_{h,k})^2$ can be used as an exit condition of \cref{gama}. In addition, with similar arguments, one can show that for any $\hat{\lambda}(\w{x}_{h,k}) \leq 0.68$ it holds $f_h(\w{x}_{h,k}) - f_h(\w{x}_h^*)  \leq \hat{\lambda}(\w{x}_{h,k})^2$.  To this end, we use $\hat{\lambda}(\w{x}_{h,k})^2 < \epsilon$ whenever the coarse direction is performed, and $\lambda(\w{x}_{h,k})^2 < \epsilon$, otherwise, to guarantee that on exit $f_h(\w{x}_{h,k}) - f_h(\w{x}_h^*) \leq \epsilon$, for some tolerance $\epsilon \in (0, 0.68^2)$.

\subsection{\blue{Quadratic Convergence Rate on the Coarse Subspace}}

In this section we show that the coarse model achieves a local quadratic convergence rate. We start with the next lemma in which we examine the required 
condition for \cref{gama} to accept the unit step. 

\begin{lemma} \label{lemma415}
Suppose that the coarse direction, $\hat{\w{d}}_{h,k}$, is employed. If 
\begin{equation*}
 \hat{\lambda}(\w{x}_{h,k}) \leq \frac{1}{2} (1 - 2 \alpha), 
\end{equation*}
where $\alpha \in (0, 1/2)$, then \cref{gama} accepts the unit step, $t_{h,k} = 1$.
\end{lemma}

Using the lemma above we shall now prove quadratic convergence of the coarse model onto the subspace spanned by the columns of $\w{R}$. The next result shows quadratic convergence when coarse steps are always taken.

\begin{lemma} \label{thm422}
Let $\lambda(\w{x}_{h,k})<1$.
 Suppose that the sequence $(\w{x}_{h,k} )_{k \in \mathbb{N}}$ is generated by \cref{gama} and $t_{h,k} = 1$. 
 Suppose also that the coarse direction, $\w{\hat{d}}_{h,k}$, is employed. Then,
 \begin{equation*}
  \hat{\lambda}(\w{x}_{h,k+1}) \leq \left( \frac{\hat{\lambda}(\w{x}_{h,k})}{1 - \hat{\lambda}(\w{x}_{h,k})} \right)^2.
 \end{equation*}
\end{lemma}

According to \cref{thm422}, we can infer the following about the convergence rate of the coarse model: first, note that the root of 
$\lambda/(1-\lambda)^2=1$ can be found at $\lambda = \frac{3 - \sqrt{5}}{2} $.  Hence, we come up with an explicit expression about the region of quadratic 
convergence, that is, when $\hat{\lambda}(\w{x}_{h,k})< \frac{3 - \sqrt{5}}{2} $, we can guarantee that 
$\hat{\lambda}(\w{x}_{h,k+1})<\hat{\lambda}(\w{x}_{h,k})$ and specifically this process converges quadratically with
\begin{equation*}
 \hat{\lambda}(\w{x}_{h,k+1}) \leq \frac{\delta}{(1-\delta)^2}\hat{\lambda}(\w{x}_{h,k}),
\end{equation*}
for some $\delta \in (0, \lambda)$. \blue{Similar to \cref{thm422}, below we show a quadratic convergence rate of the coarse model when only fine steps are taken. 
\begin{lemma} \label{thm quadratic coarse through fine}
Let $\lambda(\w{x}_{h,k})<1$.
 Suppose that the sequence $(\w{x}_{h,k} )_{k \in \mathbb{N}}$ is generated by \cref{gama} and $t_{h,k} = 1$. 
 Suppose also that the fine direction, $\w{d}_{h,k}$, is employed. Then,
 \begin{equation*}
  \hat{\lambda}(\w{x}_{h,k+1}) \leq  \frac{\lambda(\w{x}_{h,0})}{\nu (1 - \lambda(\w{x}_{h,k}))^2} \hat{\lambda}(\w{x}_{h,k})^2,
 \end{equation*}    
  where $\nu$ is defined in \cref{assumption lamda}.
\end{lemma}
The following theorem summarizes the results of \cref{thm422} and \cref{thm quadratic coarse through fine}.
\begin{theorem} \label{thm: quadratic rate of coarse both cases}
    Let $\lambda(\w{x}_{h,k})<1$. Suppose that the sequence $(\w{x}_{h,k} )_{k \in \mathbb{N}}$ is generated by \cref{gama} with $t_{h,k} = 1$. Then we obtain reduction of the approximate decrement as follows:
    \begin{enumerate}[label=(\roman*)]
        \item if the coarse step is taken then 
         \begin{equation*}
  \hat{\lambda}(\w{x}_{h,k+1}) \leq \left( \frac{\hat{\lambda}(\w{x}_{h,k})}{1 - \hat{\lambda}(\w{x}_{h,k})} \right)^2,
 \end{equation*}
        \item if the fine step is taken then
         \begin{equation*}
  \hat{\lambda}(\w{x}_{h,k+1}) \leq  \frac{\lambda(\w{x}_{h,0})}{\nu (1 - \lambda(\w{x}_{h,k}))^2} \hat{\lambda}(\w{x}_{h,k})^2,
 \end{equation*}    
  where $\nu$ is defined in \cref{assumption lamda}.
    \end{enumerate}
\end{theorem}
}

\blue{\begin{proof}
    The proof of the theorem follows directly from \cref{thm422} and \cref{thm quadratic coarse through fine}.
\end{proof}}

\blue{\cref{thm: quadratic rate of coarse both cases} provide us with a description about the convergence of
$\|\nabla f_h(\w{x}_{h,k}) \|$ onto the space spanned by the rows of $\w{R}$. As such, in the next section we examine the convergence of $\|\nabla f_h(\w{x}_{h,k}) \|$
on the entire space $\mathbb{R}^N$. }

\subsection{Composite Convergence Rate \blue{on $\mathbb{R}^n$} } \label{sec gen superlinear rate}

In this section we study the convergence of SIGMA on $\mathbb{R}^N$ and specifically we establish its composite and super-linear convergence rate for $\hat{\lambda}(\w{x}_{h,k}) \leq \eta$, {for some} $\eta>0$. We start with the following auxiliary lemma that will be useful in the discussion regarding the convergence results which follows this section.

\begin{lemma} \label{lemma norm equality lambda's d's}
For any $k \in \mathbb{N}$ it holds that
\begin{equation*}
    \left\| \ [\nabla^2 f_h ({\w{x}_{h,k}})]^{1/2} \left(
			   \w{\hat{d}}_{h,k} - \w{d}_{h,k} \right) \right\|_2 = \sqrt{\lambda(\w{x}_{h,k})^2 - 
 \hat{\lambda}(\w{x}_{h,k})^2},
\end{equation*}
where $\w{\hat{d}}_{h,k}$ as in \cref{galerkin d_h} and $\w{d}_{h,k}$ is the Newton direction.
\end{lemma}

The next lemma constitutes the core of our theorem.

\begin{lemma} \label{lemma417}
{Let $\lambda(\w{x}_{h,k})<1$.}
 Suppose that the coarse direction, $\w{\hat{d}}_{h,k}$, is employed and, in addition, that the line search selects $t_{h,k} = 1$. Then,
 {\begin{equation*}
      \lambda(\w{x}_{h,k+1}) \leq  \frac{ \sqrt{\lambda(\w{x}_{h,k})^2 - 
 \hat{\lambda}(\w{x}_{h,k})^2}}{\left( 1 - \hat{\lambda}(\w{x}_{h,k}) \right)^2} + \frac{\hat{\lambda}(\w{x}_{h,k})}{\left( 1 - \hat{\lambda}(\w{x}_{h,k}) \right)^2}
   \lambda(\w{x}_{h,k}).
 \end{equation*}}
\end{lemma}

We now use the above result to obtain the two phases of the convergence of SIGMA when coarse steps are always taken. More precisely, the region of super-linear convergence is governed by $\eta := \frac{3 - \sqrt{5 + 4\varepsilon}}{2}$, where $\varepsilon := \sqrt{1 - \mu^2}$.

\begin{lemma} \label{thm418}
 Suppose that the sequence $(\w{x}_{h,k} )_{k \in \mathbb{N}}$ is generated by \cref{gama} and that the coarse direction, 
 $\w{\hat{d}}_{h,k}$, is employed. For any $\mu \in (0,1)$, there exist constants $\gamma >0$ and $\eta \in (0, \frac{3 - \sqrt{5}}{2})$ such that 
 \begin{enumerate}[label=(\roman*)]
  \item if $\hat{\lambda}(\w{x}_{h,k}) > \eta$, then \label{case1}
  \begin{equation*}
    f_h(\w{x}_{h,k+1}) - f_h(\w{x}_{h,k}) \leq -\gamma,
  \end{equation*}
  \item if $\hat{\lambda}(\w{x}_{h,k}) \leq \eta$, then \cref{gama} selects the unit step and \label{case2}
  \begin{align}
    \hat{\lambda}(\w{x}_{h,k+1}) & < \left( \frac{\hat{\lambda}(\w{x}_{h,k})}{1 - \hat{\lambda}(\w{x}_{h,k})} \right)^2 < \hat{\lambda}(\w{x}_{h,k}). \label{reduction lambdahat} 
   \end{align}
    We further have
   \begin{equation}
          \lambda(\w{x}_{h,k+1})  < \frac{\varepsilon + \hat{\lambda}(\w{x}_{h,k})}{\left( 1 - \hat{\lambda}(\w{x}_{h,k}) \right)^2} 
   \lambda(\w{x}_{h,k}) < \lambda(\w{x}_{h,k}). \label{reduction lambda}
   \end{equation}   
 \end{enumerate}
\end{lemma}

\begin{proof}

The result in the first phase, \ref{case1}, is already proved in  \cref{lemma411} and in particular it holds for $\gamma = \alpha \beta \frac{\eta^2}{1 + \eta}$. Further, for phase \ref{case2} of the algorithm, by \cref{lemma415} and for some $\eta \in (0, \frac{3 - \sqrt{5}}{2})$ we see that \cref{gama} selects the unit step. Additionally, \cref{thm422} guarantees reduction in the approximate decrement as required by inequality \cref{reduction lambdahat}, and specifically, this process converges quadratically.  {In addition, since the coarse direction is taken, $\hat{\lambda}(\w{x}_{h,k}) > \mu \lambda(\w{x}_{h,k})$. Applying this inequality to the result of \cref{lemma417} reads
\begin{equation*}
        \lambda(\w{x}_{h,k+1}) \leq  \frac{\varepsilon + \hat{\lambda}(\w{x}_{h,k})}{\left( 1 - \hat{\lambda}(\w{x}_{h,k}) \right)^2} 
   \lambda(\w{x}_{h,k}).
\end{equation*}
Thus, convergence is achieved if the fraction in the above inequality is less than one.
}
By assumption, $\hat{\lambda}(\w{x}_{h,k}) \leq \eta$ and since $x \rightarrow \frac{\varepsilon + x}{(1-x)^2}$ is monotone increasing we obtain
\begin{equation*}
        \lambda(\w{x}_{h,k+1}) \leq  \frac{\varepsilon + \eta}{\left( 1 - \eta \right)^2} 
   \lambda(\w{x}_{h,k}).
\end{equation*}
Setting $\eta = \frac{3 - \sqrt{5 + 4\varepsilon}}{2}$ we see that $(\varepsilon + \eta)/\left( 1 - \eta \right)^2 < 1$. Finally, $\mu \in (0,1)$ implies that $\varepsilon \in (0,1)$ and thus inequality \cref{reduction lambda} holds for $\eta \in (0, \frac{3 - \sqrt{5}}{2})$ which concludes the proof of the theorem. \qed

\end{proof}

According to \cref{thm418}, for some $\eta \in (0, \frac{3 - \sqrt{5}}{2})$, we can infer the following about the convergence of \cref{gama}: In the first phase, for $\hat{\lambda}(\w{x}_{h,k}) > \eta$, the objective function is reduced as  $ f_h(\w{x}_{h,k+1}) - f_h(\w{x}_{h,k}) \leq -\gamma$ and thus the number of steps of this phase is bounded by
$\frac{1}{\gamma} [f_h(\w{x}_{h, 0}) - f_h(\w{x}_{h}^*)]$. In the second phase, for $\hat{\lambda}(\w{x}_{h,k}) \leq \eta$ the reduction is given in \cref{reduction lambda}. It is easy to show that \cref{gama} obtains a local composite convergence rate: to see this, we combine \cref{lemma upper bound lamda_hat} and \cref{reduction lambda} to get
\begin{equation*}
        \lambda(\w{x}_{h,k+1}) \leq  \frac{\varepsilon + \hat{\lambda}(\w{x}_{h,k})}{\left( 1 - \hat{\lambda}(\w{x}_{h,k}) \right)^2} 
   \lambda(\w{x}_{h,k}) \leq  \frac{\varepsilon + \lambda(\w{x}_{h,k})}{\left( 1 - \lambda(\w{x}_{h,k}) \right)^2} 
   \lambda(\w{x}_{h,k}).
\end{equation*}

{The above result shows convergence of \cref{gama} when the coarse step is always taken. In the case where \cref{gama} alternates between coarse and fine (Newton) steps, then a combined convergence behavior is expected. The complete convergence behavior of \cref{gama} is summarized in the following theorem. Note that in this case the local region of the fast convergence rate will be given according to the magnitude of the Newton decrement}.

\blue{\begin{theorem} \label{corollary complete convergence}
     Suppose that the sequence $(\w{x}_{h,k} )_{k \in \mathbb{N}}$ is generated by \cref{gama}. There exist constants $\gamma_{\text{S}}, \gamma_\text{N} >0$ and $\bar{\eta} \in (0, \frac{3 - \sqrt{5}}{2})$ such that 
 \begin{enumerate}[label=(\roman*)]
  \item if $\lambda(\w{x}_{h,k}) > \bar{\eta}$ and the coarse step is taken then $    f_h(\w{x}_{h,k+1}) - f_h(\w{x}_{h,k}) \leq -\gamma_\text{S}$ \label{cor case1}
  \item if $\lambda(\w{x}_{h,k}) > \bar{\eta}$ and the fine step is taken then $    f_h(\w{x}_{h,k+1}) - f_h(\w{x}_{h,k}) \leq -\gamma_\text{N}$ \label{cor case2}
  \item if $\lambda(\w{x}_{h,k}) \leq \bar{\eta}$ and the coarse step is taken, then \cref{gama} selects the unit step and 
  \begin{align*}
              \lambda(\w{x}_{h,k+1}) \leq  \frac{\varepsilon + \lambda(\w{x}_{h,k})}{\left( 1 - \lambda(\w{x}_{h,k}) \right)^2} 
   \lambda(\w{x}_{h,k})<\lambda(\w{x}_{h,k}),  
   \end{align*} \label{cor case3}
  \item if $\lambda(\w{x}_{h,k}) \leq \bar{\eta}$ and the fine step is taken, then \cref{gama} selects the unit step and 
  \begin{align*}
              \lambda(\w{x}_{h,k+1}) \leq  \left( \frac{\lambda(\w{x}_{h,k})}{ 1 - \lambda(\w{x}_{h,k}) } \right)^2<\lambda(\w{x}_{h,k}).  
   \end{align*} \label{cor case4}
 \end{enumerate}
\end{theorem}
}

\begin{proof}
{If $\lambda(\w{x}_{h,k})<1$ and the coarse direction is taken, we take that
\begin{equation*}
            \lambda(\w{x}_{h,k+1}) \leq  \frac{\varepsilon + \hat{\lambda}(\w{x}_{h,k})}{\left( 1 - \hat{\lambda}(\w{x}_{h,k}) \right)^2} 
   \lambda(\w{x}_{h,k}) \leq \frac{\varepsilon + \lambda(\w{x}_{h,k})}{\left( 1 - \lambda(\w{x}_{h,k}) \right)^2} 
   \lambda(\w{x}_{h,k})
\end{equation*}
where the last inequality follows since $\hat{\lambda}(\w{x}_{h,k}) \leq \lambda(\w{x}_{h,k})$. By \cref{thm418}, we obtain reduction in the Newton decrement and thus a composite convergence rate if $\lambda(\w{x}_{h,k}) \leq \eta \in (0, \frac{3 - \sqrt{5}}{2})$. On the other hand, the Newton method enters its local quadratic region of convergence if  $\lambda(\w{x}_{h,k}) \leq \eta_\text{N} \in (0, \frac{3 - \sqrt{5}}{2})$ \cite{MR2142598}. Setting $\bar{\eta} := \min\{ \eta, \eta_\text{N} \}$ and using $\lambda(\w{x}_{h,k}) \leq \bar{\eta}$ proves items \ref{cor case3} and \ref{cor case4}.}

{
Further, in \cref{cor case1}, since the coarse direction is taken it holds that $\hat{\lambda}(\w{x}_{h,k}) \geq \mu \lambda(\w{x}_{h,k}) > \mu \eta$. Combining this inequality with the result of \cref{lemma411}, the result holds for $\gamma_\text{S} = \alpha \beta \frac{\mu^2 \eta^2}{1 + \mu \eta}$. Lastly, \cref{cor case2} holds for $\gamma_\text{N} = \alpha \beta \frac{ \eta_\text{N}^2}{1 + \eta_\text{N}}$ (see \cite{MR2061575}). \qed }
\end{proof}

\subsection{Discussion on the Convergence Results} \label{sec discussion on thm}

Notice that the two phases of \cref{thm418} depend on the user-defined parameter $\mu$. That is, the region of the composite convergence is proportional to the value of $\mu$ that satisfies \cref{assumption lamda}. Specifically, as $\mu \rightarrow 1$, we see that $\varepsilon \rightarrow 0$ and  $\eta \in (0, \frac{3 - \sqrt{5}}{2})$ and thus SIGMA approaches the fast convergence of the full Newton method. On the other hand, as $\mu \rightarrow 0$, \cref{thm418} indicates a restricted region of the second phase and thus slower convergence. Therefore, as a consequence of \cref{thm418}, the user is able to select a-priori the desired region of the fast convergence rate through $\mu$. However, bear in mind that larger values in the user-defined parameter $\mu$ may yield more expensive iterations (fine steps). Therefore, there is a trade-off between the number of coarse steps and the choice of $\mu$. {Note also that the theorem shows that it is possible for the method to achieve a super-linear convergence rate. This will be attained if we select iteration-dependent $\w{R}_{k}$ such that $\varepsilon_k$ converges to zero as $k$ goes to zero. We examine this case in the next section}.
The following lemma offers further insights on the composite convergence rate. 

\begin{lemma} \label{lemma norm on d's upper bound}
For all $k \in \mathbb{N}$ it holds that
\begin{equation*}
    \left\| [\nabla^2 f_h ({\w{x}_{h,k}})]^{1/2} \left(\w{\hat{d}}_{h,k} - \w{d}_{h,k} \right) \right\|_2 \leq \lambda(\w{x}_{h,k}).
\end{equation*}
\end{lemma}

The above result shows that, without using \cref{assumption lamda}, the quantity of interest in \cref{lemma417} can, in the worst case, be as much as the Newton decrement. However, if $\hat{\lambda}(\w{x}_{h,k}) \leq \eta$, to obtain reduction in the Newton decrement, we must have $\left\| [\nabla^2 f_h ({\w{x}_{h,k}})]^{1/2} \left(\w{\hat{d}}_{h,k} - \w{d}_{h,k} \right) \right\|_2 \leq \varepsilon \lambda(\w{x}_{h,k})$, where $0 \leq \varepsilon < 1$. Analyzing this inequality for $\varepsilon=1$ we identify the following cases:
\begin{enumerate}[label=(\roman*), leftmargin=*]
\item \label{case1 equality}  $ \w{\hat{d}}_{h,k} = c \w{d}_{h,k} $, where $\varepsilon = 1$ holds for $c=0$ or $c=2$. To see this, by \cref{lemma norm on d's upper bound} we take,
\begin{equation*}
    \left\| [\nabla^2 f_h ({\w{x}_{h,k}})]^{1/2} \left( c-1 \right)\w{d}_{h,k} \right\|_2 = \left| c-1 \right|\lambda(\w{x}_{h,k}) \leq \lambda(\w{x}_{h,k}), 
\end{equation*}
and thus $\varepsilon=1$ holds for $c=0$ or $c=2$. The case $c=0$ indicates that $\w{\hat{d}}_{h,k} = 0$ while $\w{d}_{h,k} \neq 0$ as also discussed in \cref{sec universal method}. Alternatively, $c=0$ implies $\w{\hat{d}}_{h,k} = \w{d}_{h,k} = 0$ which is attained in limit. Further, $c=2$ also implies $ \left\| [\nabla^2 f_h ({\w{x}_{h,k}})]^{1/2} \left(\w{\hat{d}}_{h,k} - \w{d}_{h,k} \right) \right\|_2 \leq \lambda(\w{x}_{h,k})$, however, we note that this is an extreme case that rarely holds in practice.
\item \label{case2 equality lambdas} Similarly, we have that $\hat{\lambda}(\w{x}_{h,k}) = 0$ while $\lambda(\w{x}_{h,k}) \neq 0$ or $\hat{\lambda}(\w{x}_{h,k}) = \lambda(\w{x}_{h,k}) = 0$. This can be derived directly from \cref{lemma norm equality lambda's d's} and it is exactly the case \ref{case1 equality} for $c=0$.
\end{enumerate}
When either of the above cases hold, the multilevel algorithm will not progress. Notice that composite convergence rate holds for any choice of the prolongation operator and thus, since we make no further assumptions, this result verifies our intuition, i.e., there exists a choice of $\w{P}$ that may lead to an ineffective coarse step. Therefore, to prevent this, we make use of the \cref{assumption lamda}. Nevertheless, we notice that the above cases rarely hold in practice and thus we should expect the multilevel algorithm to enter the second-phase near the minimizer. {Additionally, the above cases indicate that in the worst case scenario the multilevel method will converge to a sub-optimal solution. However, it will never diverge. This result is important considering we take no assumptions on $\w{P}$ whatsoever.}

As discussed above, in this setting, SIGMA requires checking \cref{assumption lamda} (or $\| \w{R} \nabla f (\w{x}_{h,k}) \|_2 > \mu \| \nabla f (\w{x}_{h,k}) \|_2 $, see \cref{remark gama}) at each iteration. However, the checking process can be expensive for large scale optimization. In the next section we present a randomized version of the multilevel algorithm and we study problem instances where the checking process can be omitted. {In such cases there exists $0 \leq \varepsilon < 1$  and thus SIGMA is guaranteed to always converge to the solution even if fine steps are never taken. This result is new when analyzing multilevel methods and it further improves the efficacy of the proposed algorithm for large-scale optimization.}

\subsection{Super-linear Convergence Rate through the Nystr\"om Method} \label{sec nystrom}

The computational bottlenecks in  \cref{gama} are: \textbf{(i)} the construction of the coarse step in \cref{galerkin d_h}, and \textbf{(ii)} the checking process in order to avoid an ineffective search direction. To overcome both, we select the prolongation operator as in \cref{def P}, thus the Galerkin model is generated based on the low-rank approximation of the Hessian matrix through the Nystr\"om method. {In particular, we show that if SIGMA enters the super-linear phase, then the checking process can be omitted.} Moreover, performing the Nystr\"om method we are able to overcome the computational issues related to the construction of the coarse direction (see \cref{sec nystrom intro}, \cref{remark Q_H nystrom} and \cref{remark reduced hess} for details on the efficient implementation of \cref{gama} using the Nystr\"om method). 

{
Given $\w{R}_k$ as in \cref{def P}, let $\rho \in [0,1]$ be the probability that $\hat{\lambda}(\w{x}_{h,k}) \leq \nu \in (0,0.68^2)$,  for any $k \in \mathbb{N}$ such that $\w{x}_{h,k} \neq \w{x}_h^*$. Recall from the results of the previous section that
$\hat{\lambda}(\w{x}_{h,k}) =  0$ implies no progress for the multilevel scheme and thus convergence to a sub-optimal solution. For this reason \cref{assumption lamda} was necessary in our previous analysis. Given now the fact that $\w{R}_k$ is constructed randomly according to \cref{def P}, it is possible to abandon \cref{assumption lamda} and provide probabilistic results that arise from some discrete probability distribution.
Note that by \cref{definition lamdahat and lamda}, $\rho$ effectively is the probability of selecting sufficiently small partial derivatives. Given the sampling strategies introduced in \cref{sec nystrom intro}, $\rho$ is expected to be small enough or zero. 
}

{
\begin{lemma} \label{lemma bounds unknown quantity}
Assume that $\w{R}_k$, $k \in \mathbb{N}$, is constructed as in \cref{def P}. Then, there exists $\mu_k \in \left(0, \frac{\hat{\lambda}(\w{x}_{h,k})}{\lambda(\w{x}_{h,k})} \right] $ such that $ 0 < \mu_k \leq 1$ and
\begin{align*} \label{eq bounds unknown quantity 1}
    \sqrt{ \lambda(\w{x}_{h,k})^2 - \hat{\lambda}(\w{x}_{h,k})^2} & \leq \sqrt{1 - \mu_k^2} \lambda(\w{x}_{h,k}) 
\end{align*}
with probability $1-\rho$.
\end{lemma}
}

{Since there exist iteration-dependent scalars $\mu_k$, we are able to show a super-linear convergence rate. This is true because $0<\mu_k\leq1$, and thus it is possible to select $\w{R}_k$ such that $\lim_{k \rightarrow \infty} \mu_k = 1$. Therefore, to show the desired result, we additionally need to impose the assumption that  $\lim_{k \rightarrow \infty} \mu_k = 1$.
The theorem below presents an instance of SIGMA that achieves a local super-linear rate. 
}

{
\begin{theorem} \label{thm seuperlinear probabilities}
Suppose that the coarse direction is constructed with $\w{R}_k$ as in \cref{def P} such that $\mu_k = 1 - \frac{1}{2 \ln(2 + k)}$. Moreover, suppose that the sequence $( \w{x}_{h,k} )_{k \in \mathbb{N}}$ is generated by \cref{gama}. Then, there exist constants $\hat{\gamma} >0$ and $\hat{\eta} \in (0, \frac{3 - \sqrt{5}}{2})$ such that 
 \begin{enumerate}[label=(\roman*)]
  \item if $\lambda(\w{x}_{h,k}) > \hat{\eta}$, then, with probability $1-\rho$,  \label{case1 probs}
  \begin{equation*}
    f_h(\w{x}_{h,k+1}) - f_h(\w{x}_{h,k}) \leq -\hat{\gamma},
  \end{equation*}
  \item if $\lambda(\w{x}_{h,k}) \leq \hat{\eta}$, then \cref{gama} selects the unit step and \label{case2 probs}
  \begin{align*}
    \hat{\lambda}(\w{x}_{h,k+1}) & < \left( \frac{\hat{\lambda}(\w{x}_{h,k})}{1 - \hat{\lambda}(\w{x}_{h,k})} \right)^2 < \hat{\lambda}(\w{x}_{h,k}) 
   \end{align*}
   where this process converges quadratically. Setting $\varepsilon_k := \sqrt{1 - \mu_k^2}$, we further have
   \begin{equation*}
          \lambda(\w{x}_{h,k+1})  < \frac{\varepsilon_k + \lambda(\w{x}_{h,k})}{\left( 1 - \lambda(\w{x}_{h,k}) \right)^2} 
   \lambda(\w{x}_{h,k}) < \lambda(\w{x}_{h,k}) \label{reduction lambda probs}
   \end{equation*}
   where this process achieves a super-linear convergence rate.  Both results in this phase hold with probability $1 - \rho$.
 \end{enumerate}
\end{theorem}
}

{\begin{proof}
    Recall the following inequality from \cref{lemma411}
    \begin{equation*}
 f_h(\w{x}_{h,k} + t_{h,k} \w{\hat{d}}_{h,k}) - f_h(\w{x}_{h,k}) \leq - \alpha \beta \frac{\hat{\lambda}(\w{x}_{h,k})^2}{1 + \hat{\lambda}(\w{x}_{h,k})}.
\end{equation*}
Then from \cref{lemma bounds unknown quantity} we have that $\hat{\lambda}(\w{x}_{h,k}) \geq \mu_k \lambda(\w{x}_{h,k}) \geq \mu_k \hat{\eta} \geq \mu_0 \hat{\eta} $, where the last equality holds since $\mu_k$ forms an increasing sequence of real numbers. 
Thus, using $\hat{\gamma} := \alpha \beta \frac{\mu_0^2 \hat{\eta}^2}{1 + \mu_0 \hat{\eta}}$ and since $x \rightarrow \frac{x^2}{1+x}$ is an increasing function, the result of the first phase follows with probability $1 - \rho$. Let $K \in \mathbb{N}$ denote the first iteration that satisfies $\lambda(\w{x}_{h,K})<1$. Then, for all $k \geq K$, we have that $\mu_K \lambda(\w{x}_{h,k}) \leq \mu_k \lambda(\w{x}_{h,k}) \leq \hat{\lambda}(\w{x}_{h,k}) \leq \lambda(\w{x}_{h,k}) < 1$, where, in fact, $\mu_K = 1 - \frac{1}{\ln(2 + K)}$ is the smallest of such $\mu_k$.
Combining the last inequality, \cref{lemma bounds unknown quantity} and \cref{lemma415}, the method accepts the unit step if 
$\lambda(\w{x}_{h,k}) \leq \frac{1 - 2\alpha}{2 \mu_K}$. This proves \cref{thm422} with probability $1-\rho$. Last, as in the proof of \cref{thm418}, for $\lambda(\w{x}_{h,k}) < 1$, we can show
\begin{equation*}
        \lambda(\w{x}_{h,k+1}) \leq  \frac{\varepsilon_k + \hat{\lambda}(\w{x}_{h,k})}{\left( 1 - \hat{\lambda}(\w{x}_{h,k}) \right)^2} 
   \lambda(\w{x}_{h,k}) \leq  \frac{\varepsilon_k + \lambda(\w{x}_{h,k})}{\left( 1 - \lambda(\w{x}_{h,k}) \right)^2} 
   \lambda(\w{x}_{h,k}),
\end{equation*}
with probability $1-\rho$. Here, the difference is that the scalar $\varepsilon_k$ is not fixed but it depends on the iteration-dependent $\mu_k$, i.e., $\varepsilon_k = \sqrt{1 - \mu_k^2}$. Therefore,
\begin{equation*}
    \lim_{k \rightarrow \infty} \frac{\varepsilon_k + \lambda(\w{x}_{h,k})}{\left( 1 - \lambda(\w{x}_{h,k}) \right)^2} =   0,
\end{equation*}
which proves the super-linear convergence. Finally, the scalar $\hat{\eta}$ that controls the super-linear rate is given by $\hat{\eta} := \frac{3 - \sqrt{5 + 4 \varepsilon_K}}{2}$, where $\varepsilon_K = \sqrt{1 - \mu_K^2}$ and thus we obtain $\hat{\eta} \in (0, \frac{3 - \sqrt{5}}{2})$. Then, we can ensure that for all $k \in \mathbb{N}$ such that $\lambda(\w{x}_{h,k}) < \hat{\eta}$ it holds $\lambda(\w{x}_{h,k+1}) < \lambda(\w{x}_{h,k})$, and hence in the second-phase \cref{gama} converges super-linearly. \qed
\end{proof} 
}

{The above theorem shows that the multilevel algorithm can achieve a super-linear convergence if $\w{R}_k$ is selected such that $\lim_{k \rightarrow \infty} \mu_k = 1$. In this case, fine directions need not be taken during the process. In the next section, we illustrate that SIGMA can achieve a super-linear or even a quadratic convergence rate in practice, and thus the key assumption of \cref{thm seuperlinear probabilities} is not just theoretical. The theorem also allows for quadratic convergence rate, i.e., there exists $K_q \in \mathbb{N}$ such that $\mu_k = 1$ for all $k \geq K_q$.
Quadratic rates are also observed for SIGMA in our experiments. Therefore, our theory is consistent with practice. Moreover, obviously, $\lim_{k \rightarrow \infty} \mu_k = 1$ holds when $\lim_{k \rightarrow \infty} \hat{\lambda}(\w{x}_{h,k}) = \lambda(\w{x}_{h,k})$. This condition is expected to be satisfied for $ n \rightarrow N$. However, larger values in $n$ yield more expensive iterations. Therefore, one should search for examples where $\hat{\lambda}(\w{x}_{h,k}) \approx \lambda(\w{x}_{h,k})$ for $n \ll N$. Instances of problem structures that satisfy $\hat{\lambda}(\w{x}_{h,k}) \approx \lambda(\w{x}_{h,k})$ include cases such as when the Hessian matrix is nearly low-rank or when there is a big gap between its eigenvalues, i.e., $\lambda_1 \geq \lambda_2 \geq \cdots \geq \lambda_{n} \gg \lambda_{n+1} \geq \cdots \geq \lambda_N$, or, in other words, when the important second-order information is concentrated in the first few eigenvalues. Given such problem structures, we expect very fast convergence rates without ever using fine directions, and thus a convergence behavior similar to \cref{thm seuperlinear probabilities}. This claim is also verified via numerical experiments.}

\blue{
On the other hand, in the absence of the assumption $\lim_{k \rightarrow \infty} \mu_k = 1$, we can still guarantee convergence of SIGMA with a local composite rate and probability $1 - \rho$. However, in this case, the assumption that there exists a global $\mu \in (0,1]$ such that $\hat{\lambda}(\w{x}_{h,k}) \geq \mu \lambda(\w{x}_{h,k})$ with probability $1-\rho$ is required. We may as well compute the total number of steps that \cref{gama} requires to reach a tolerance $\epsilon$. Assume that the algorithm enters the composite phase and achieves a rate $a_r$, i.e., $\lambda(\w{x}_{h,{k+1}}) \leq a_r \lambda(\w{x}_{h,k})$. We also can obtain the region of composite convergence which depends on $a_r$, that is $\eta(a_r)$. We set $\eta := \eta(a_r)$ for simplicity. Given the above considerations we can guarantee the following about the total number of steps in both phases.
}

\blue{
\begin{corollary} \label{corol: no of iterations}
    Let $\gamma := \alpha \beta \frac{\eta^2}{1 + \eta}$, $\eta := \eta(a_r)$ and $\epsilon \in (0, 0.68^2)$.
    Further, suppose that the coarse direction is constructed with $\w{R}_k$ as in \cref{def P} and that the sequence $( \w{x}_{h,k} )_{k \in \mathbb{N}}$ is generated by \cref{gama}. Then, the total number of iteration for the \cref{gama} to reach an $\epsilon$-approximate value in the objective function is at least
    \begin{equation*}
        N_\epsilon = \left \lceil{ \frac{f(\w{x}_{h,0} - \w{x}_{h}^* )}{\gamma}} \right \rceil  + \left \lceil{ \frac{\log_b(\frac{\epsilon}{\eta})}{\log_b(a_r)}} \right \rceil, \quad b >0, b \neq 1,
    \end{equation*}
    with probability $1-\rho$.
\end{corollary}
}

\begin{proof}
\blue{    Recall that the main assumption we use here is that $\hat{\lambda}(\w{x}_{h,k}) \geq \mu \lambda(\w{x}_{h,k})$ which holds with probability $1-\rho$. Then, we guarantee that $\hat{\lambda}(\w{x}_{h,k}) > 0$ and thus the result of the first phase of \cref{thm418} holds with probability $1-\rho$. This shows that the the objective function decreases by at least $\gamma$ at each iteration and thus the number of iteration taken by \cref{gama} during the first phase is at most 
    \begin{equation*}
       \frac{f(\w{x}_{h,0} - \w{x}_{h}^* )}{\gamma},
    \end{equation*}
    with probability $1-\rho$. }
    \blue{
    Moreover, assume that at an iteration $K \in \mathbb{N}$ the algorithm enters the composite phase. It holds $\lambda(\w{x}_{h,K}) \leq \eta$. Next, let $K_{\epsilon}$ be the number of iterations required for the method to reach a tolerance $\epsilon$. Then by \cref{lemma suboptimality gap}
    \begin{align*}
      f(\w{x}_{h,K + K_\epsilon}) - f(\w{x}_{h}^*) & \leq  \lambda(\w{x}_{h,K + K_\epsilon})^2 \leq a_r^{2 K_\epsilon} \lambda(\w{x}_{h,K})^2 \leq a_r^{2 K_\epsilon} \eta^2,
    \end{align*}
    where the second inequality holds by \cref{thm418} assuming  that the Newton decrement decreases as $\lambda(\w{x}_{h,{k+1}}) \leq a_r \lambda(\w{x}_{h,k})$. Note that this decrease holds with probability $1-\rho$ due to the main assumption. Therefore, $a_r^{2 K_\epsilon} \eta^2 \leq \epsilon^2$ yields
    \begin{equation*}
        K_\epsilon \geq \frac{\log_b(\frac{\epsilon}{\eta})}{\log_b(a_r)},
    \end{equation*}
    which ensures that $f(\w{x}_{h,K + K_\epsilon}) - f(\w{x}_{h}^*) \leq \epsilon^2$ after $K_\epsilon + 1$ iterations in the composite phase. Putting this all together the result of the theorem follows with probability $1-\rho$. \qed
    } 
\end{proof}
\blue{
Below we present a simple example using an explicit value on $a_r$ (and therefore on $\eta$ and $\mu$). 
}

\blue{
\emph{Example: Assume $a_r := 1/2$, i.e., we wish the decrease to be as much as $\lambda(\w{x}_{h,{k+1}}) \leq 1/2 \lambda(\w{x}_{h,k})$ in the second phase of \cref{thm418}. Applying the proof of \cref{thm418}, this value of $a_r$ can be achieved as long as  $\mu > \sqrt{3/4}$. Moreover, $a_r = 1/2$ yields the following definition on the region of the composite convergence: $\eta = 2 - \sqrt{3 + 2 \varepsilon}$, where $\varepsilon$ as in \cref{thm418}. Note that the condition  $\mu \in (\sqrt{3/4}, 1]$ ensures that $\eta$ is positive. Specifically, we have that $\eta \in (0, 13/50)$. Then, applying these values in the result of \cref{corol: no of iterations}, and given a tolerance $\epsilon := 10^{-5}$, we obtain $K_\epsilon \geq  
 \left \lceil{ \log_2(\frac{13}{50 \epsilon}}) \right \rceil = 15$. Therefore, in this set up, with a probability $1-\rho$, \cref{gama} requires just $15$ iterations during the composite phase to ensure that $f(\w{x}_{h,K + K_\epsilon}) - f(\w{x}_{h,*}) \leq \epsilon^2$ on exit, using only coarse steps.}
}

Furthermore, in the remark below, we give insights on the magnitude of probability $\rho$ when the sampling strategy arises from the uniform or adaptive distributions (see \cref{sec nystrom intro}).

{
\begin{remark} \label{remark on probability rho}
Recall that according to the \cref{def P}, the probability that $\hat{\lambda}(\w{x}_{h,k}) \leq \nu$, boils down to the probability of selecting the $n$-(almost) zero entries from the $N$ entries of the gradient vector, while $\| \nabla f(\w{x}_{h,k}) \| \neq 0$. Let $r$ be the number of zero elements of the gradient vector and assume that the $n$ samples are drawn uniformly without replacement. Then, if $n > r$ we take $\rho = 0$, and thus the results in \cref{thm seuperlinear probabilities} hold with probability one. If $n \leq r$, then ${\binom{r}{n}}$ denotes all the $n$ combinations of the zero elements, and ${\binom{n}{N}}$ the total number of $n$ combinations of the $N$ entries of the gradient vector. Then,
\begin{equation*}
    \rho = \frac{{\binom{r}{n}}}{{\binom{N}{n}}} = \frac{(r - n +1)  \cdots (r-1)r}{(N - n +1) \cdots (N-1)N}.
\end{equation*}
This result indicates that when $r \ll N$, $\rho$ must be small enough. In practical applications, it is common to expect $r \ll N$ or $n > r$, and hence SIGMA will converge with high probability or probability one, respectively. On the other hand, if the $n$ samples are collected according to the adaptive distribution then, by construction, $\rho = 0$ and SIGMA will converge with probability one. The aforementioned observation is presented formally in the corollary below, and it will be verified through the numerical experiments in \cref{sec experiments} and \cref{sec num res}. \qed
\end{remark}
}

\blue{
\begin{corollary}
Suppose that the coarse direction is constructed with $\w{R}_k$ as in \cref{def P}, where $\mathbf{p}$ is the adaptive distribution (see \cref{sec nystrom intro}). Assume also that $\mu_k = 1 - \frac{1}{2 \ln(2 + k)}$. Moreover, suppose that the sequence $( \w{x}_{h,k} )_{k \in \mathbb{N}}$ is generated by \cref{gama}. Then, there exist constants $\hat{\gamma} >0$ and $\hat{\eta} \in (0, \frac{3 - \sqrt{5}}{2})$ such that 
 \begin{enumerate}[label=(\roman*)]
  \item if $\lambda(\w{x}_{h,k}) > \hat{\eta}$, then
  \begin{equation*}
    f_h(\w{x}_{h,k+1}) - f_h(\w{x}_{h,k}) \leq -\hat{\gamma},
  \end{equation*}
  with probability one,
  \item if $\lambda(\w{x}_{h,k}) \leq \hat{\eta}$, then \cref{gama} selects the unit step and 
  \begin{align*}
    \hat{\lambda}(\w{x}_{h,k+1}) & < \left( \frac{\hat{\lambda}(\w{x}_{h,k})}{1 - \hat{\lambda}(\w{x}_{h,k})} \right)^2 < \hat{\lambda}(\w{x}_{h,k}) 
   \end{align*}
   where this process converges quadratically. Setting $\varepsilon_k := \sqrt{1 - \mu_k^2}$, we further have
   \begin{equation*}
          \lambda(\w{x}_{h,k+1})  < \frac{\varepsilon_k + \lambda(\w{x}_{h,k})}{\left( 1 - \lambda(\w{x}_{h,k}) \right)^2} 
   \lambda(\w{x}_{h,k}) < \lambda(\w{x}_{h,k})
   \end{equation*}
   where this process achieves a super-linear convergence rate.  Both results in this phase hold with probability one.
 \end{enumerate}
\end{corollary}
}
\blue{
\begin{proof}
Combining \cref{thm seuperlinear probabilities} and the definition of the adaptive sampling strategy in \cref{sec nystrom intro}, the result of the theorem follows immediately. \qed   
\end{proof}
}

\blue{
Last, we prove a simple convergence result of SIGMA in expectation. It shows that the expected value of the objective function decreases by at least $\tilde{\gamma} > 0$ at each iteration.
}

\blue{
\begin{theorem} \label{thm expects}
 Suppose that we select $\w{R}_k$ as in \cref{def P} and that the sequence $ ( \w{x}_{h,k} )_{k \in \mathbb{N}}$ is generated from \cref{gama}. Suppose also $\hat{\lambda}(\w{x}_{h,k}) > \nu$ while $\w{x}_{h,k} \neq \w{x}_{h}^*$. Then, there exists $\tilde{\gamma} > 0$ such that
  \begin{equation*}
    \mathbb{E} \left[ f_h(\w{x}_{h,k+1}) - f_h(\w{x}_{h,k}) \right]\leq -\tilde{\gamma}.
  \end{equation*}
\end{theorem}
}

\begin{proof}
Since $\hat{\lambda}(\w{x}_{h,k}) > \nu$, \cref{lemma bounds unknown quantity} holds with probability one. Then, taking expectation conditioned on $\mathcal{F}_k$ we have that $     \tilde{\mu}^2 \leq  \frac{ \mathbb{E} [ \hat{\lambda}(\w{x}_{h,k}) | \mathcal{F}_k ]^2}{\lambda(\w{x}_{h,k})^2}$, where $\tilde{\mu} := \mathbb{E}[\mu_k | \mathcal{F}_k]>0$. Moreover, from \cref{lemma411} we have that
\begin{equation*}
f_h(\w{x}_{h,k+1}) - f_h(\w{x}_{h,k}) \leq - \alpha \beta \frac{\hat{\lambda}(\w{x}_{h,k})^2}{1 + \hat{\lambda}(\w{x}_{h,k})}.
\end{equation*}
Let $g : x \rightarrow \frac{x^2}{1+x}$, $\operatorname{dom} g = [0,1)$, and take expectation on both sides conditioned on $\mathcal{F}_k$. Then, $ \mathbb{E} \left[ f_h(\w{x}_{h,k+1}) | \mathcal{F}_k \right] - f_h(\w{x}_{h,k}) \leq - \alpha \beta \mathbb{E} \left[ g (\hat{\lambda}(\w{x}_{h,k})) | \mathcal{F}_k \right]$. By convexity of $g$ and Jensen's inequality we take $ \mathbb{E} \left[ f_h(\w{x}_{h,k+1}) | \mathcal{F}_k \right] - f_h(\w{x}_{h,k}) \leq - \alpha \beta g \left( \mathbb{E} \left[ \hat{\lambda}(\w{x}_{h,k}) | \mathcal{F}_k \right] \right)$.  Using this, the fact that $g$ is monotone increasing and 
 $\mathbb{E} \left[ \hat{\lambda}(\w{x}_{h,k}) | \mathcal{F}_k \right] > \tilde{\mu} \lambda(\w{x}_{h,k})$, we take
\begin{equation*}
    \mathbb{E} \left[ f_h(\w{x}_{h,k+1}) | \mathcal{F}_k \right] - f_h(\w{x}_{h,k}) \leq - \alpha \beta \frac{\hat{\mu}^2 \lambda(\w{x}_{h,k})^2}{1 + \hat{\mu} \lambda(\w{x}_{h,k})}
\end{equation*}
which holds almost surely. Taking expectation on both sides w.r.t. the randomness induced by \cref{def P}, the claim follows with $\tilde{\gamma} = \alpha \beta \mathbb{E} \left[ \frac{\hat{\mu}^2 \lambda(\w{x}_{h,k})^2}{1 + \hat{\mu} \lambda(\w{x}_{h,k})} \right] >0$. Last, one can also show a more conservative bound using $\mathbb{E} \left[ \hat{\lambda}(\w{x}_{h,k}) | \mathcal{F}_k \right] > \nu$. Then, we take $\tilde{\gamma} = \alpha \beta  \frac{\nu^2}{1 + \nu}$. \qed
\end{proof}

{The above result, even though it does not discusses the convergence rate in expectation, is important as it indicates that, when the coarse directions are constructed randomly, we can always find $\mu>0$ such that, on average, \cref{assumption lamda} is satisfied. This now means that convergence to the global minimum is attained as long as $\hat{\lambda}(\w{x}_{h,k}) > \nu$. Hence, \cref{thm expects} and \cref{thm seuperlinear probabilities} are complementary to each other and together indicate that the coarse directions, constructed from \cref{def P}, will always be effective and yield progress of \cref{gama}. Thus, convergence of \cref{gama} to the minimum should be expected without ever taking fine directions. This important result is obtained since the method enjoys a global convergence analysis and will not diverge, see \cref{corollary complete convergence} and discussion in \cref{sec discussion on thm}}.

\section{Examples and Numerical Results} \label{sec experiments}

In this section we validate the efficacy of the proposed algorithm and we verify our theoretical results 
on optimization problems that arise in machine learning applications. Specifically, in the first two sections {we use the conventional SIGMA (uniform sampling)} to solve the maximum likelihood estimation problem based on the Poisson and Logistic models, respectively. Full details about the experimental setup, objective functions and the datasets are given in \cref{sec num res}. {Furthermore, in \cref{sec exps sample strategies} we provide comparisons between the conventional SIGMA and SIGMA with the different sampling strategies of \cref{sec nystrom intro}}. Moreover, we provide additional experiments and we test SIGMA with sub-sampling (\cref{sec sub-sampled sigma}). In \cref{remark reduced hess} we discuss how to efficiently compute the reduced Hessian matrix for Generalized Linear Models.

\subsection{Poisson Regression and Impact of the Spectral Gap}

\begin{figure}[t]
	\begin{subfigure}{.6\textwidth}
		\centering
		\includegraphics[width=1\linewidth]{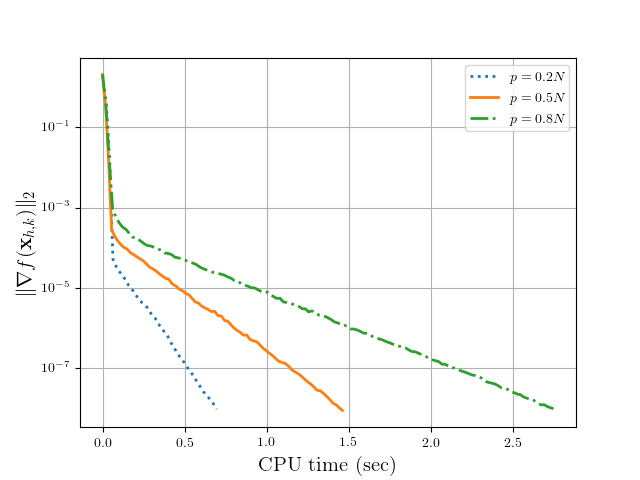}
		\caption{Synthetic}
		\label{subfig kappa}
	\end{subfigure}
	\begin{subfigure}{.6\textwidth}
		\centering
		\includegraphics[width=1\linewidth]{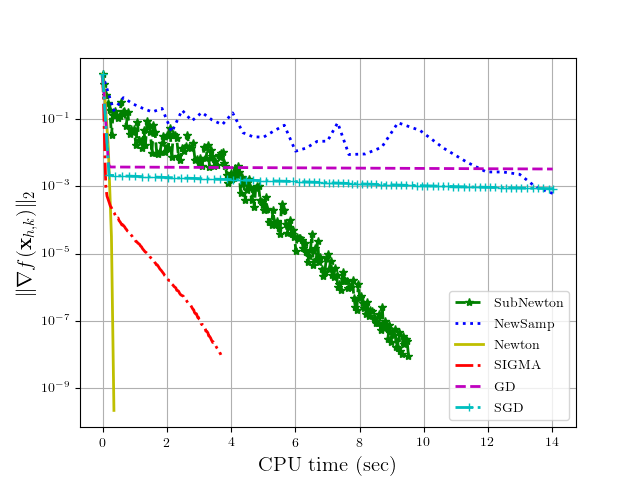}
		\caption{Synthetic}
		\label{subfig gen}
	\end{subfigure}
	\begin{subfigure}{.6\textwidth}
		\centering
		\includegraphics[width=1\linewidth]{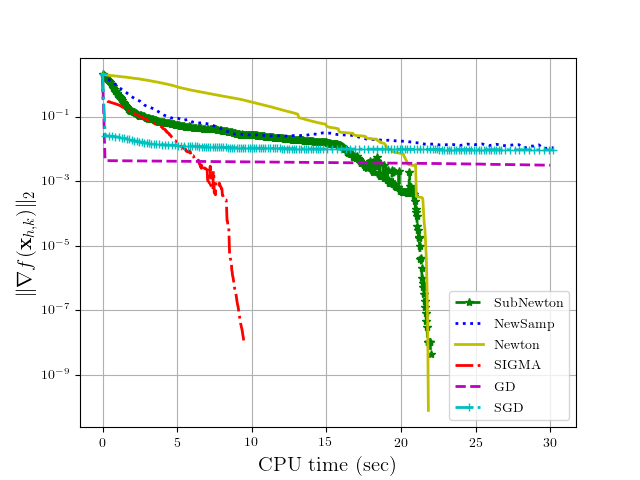}
		\caption{Synthetic with $\xi_1 = 10^{-3}$}
		\label{subfig genl1}
	\end{subfigure}
	\caption{From left to right, the first figure shows convergence of SIGMA for different values in $p$. The second and third figures compare the optimization methods over the $\ell_2$ and elastic-net regularized Poisson regression respectively.}
	\label{fig poisson}
\end{figure}

In this section we attempt to verify the claim that SIGMA enjoys a very fast super-linear convergence rate when there is a big gap between the eigenvalues of the Hessian matrix (see discussion that follows \cref{thm seuperlinear probabilities}). To illustrate this, we generate an input matrix $\w{A}$ based on the Singular Value Decomposition (SVD) and then we run experiments on the Poisson model. More precisely, $\w{A}$ is constructed as follows. We first generate random orthogonal matrices $\w{U} \in \mathbb{R}^{m \times m}$ and $\w{V} \in \mathbb{R}^{N \times N}$, $m > N$, from the $\mathcal{O}(N)$ Haar distribution \cite{mezzadri2006generate}. Denote now the matrix of singular values as $
\w{\Sigma} := 
\begin{bmatrix}
\w{\Sigma}_N & \w{0} \\
\end{bmatrix}^T \in  \mathbb{R}^{m \times N}$,
where $\w{\Sigma}_N := \operatorname{diag} (\sigma_1, \sigma_2, \ldots, \sigma_N)$ is a square diagonal matrix. To form a gap, we select some $p$ from the set of integers  $S_N = \{1,2, \ldots, N \} $ and then we compute $N$ singular values such that  $\sigma_1 > \sigma_2 > \cdots > \sigma_p \gg \sigma_{p+1} \geq \cdots \geq \sigma_N>0$, where the first $p$ are evenly-spaced. Then, we set $\w{A} := \w{U} \w{\Sigma} \w{V}^T$ and we expect $\w{A}$ to have a big gap between the $p$ and ${p+1}$ singular values. Further, note that the Hessian matrix of the Poisson model contains the product $\w{A}^T \w{A}$. Observe that $\w{A}^T \w{A} = \w{V} \w{\Sigma}_N^2 \w{V}^T$, and thus $\sigma_1^2, \sigma_2^2, \ldots, \sigma_N^2$ are the eigenvalues of $\w{A}^T \w{A}$. Therefore, by the construction of $\w{\Sigma}_N$, we should expect the Hessian matrix to have a big gap between the $p$ and ${p+1}$ eigenvalues. 

In the first experiment (\cref{subfig kappa}), we compare the performance of SIGMA for three different locations of the ``singular value gap'', i.e., $p = \{ 0.2N, 0.5 N, 0.8N \} $. In all three cases we set  the coarse model dimensions to $n = N/2$. Figure \ref{subfig kappa} shows the effect of the eigenvalue gap in the convergence rate of SIGMA. Clearly, when the gap is placed in the first few singular values (or eigenvalues respectively) SIGMA achieves a very fast super-linear rate. In particular, when $p = 0.2N$,  the convergence of SIGMA to the solution is about five times faster in comparison to the convergence when $p=0.8N$, which  verifies our intuition that smaller $p$ yields faster convergence rates for SIGMA. Similar behavior should be expected when the Hessian matrix is (nearly) low-rank
or when important second-order information is concentrated on the first few eigenvalues.

Next, we compare SIGMA against the other optimization methods when $\w{A}$ is generated with $p = 0.5N$. The coarse model dimensions for SIGMA is set as above. For the NewSamp and SubNewton we use $|S_m| = m/2$ samples at each iteration. Figure \ref{subfig gen} shows the performance between the optimization methods over the $\ell_2$-regularized Poisson regression. We observe that in the first phase, both the Newton method and SIGMA have similar behavior and they move rapidly towards the solution. Then, the Newton method enters in its quadratic phase and converges in few iterations while SIGMA achieves a super-linear rate. Besides the Newton method and SIGMA, SubNewton is the only algorithm that is able to reach a satisfactory tolerance within the time limit, but it is much slower. 

Enforcing sparsity in the solution is typical in most imaging applications that employ the Poisson model. Thus, we further run experiments with the pseudo-Hubert function in action with $\xi_1 = 10^{-3}$. The algorithms set up is as described above. The results for this experiment are reported in \cref{subfig genl1}.  Clearly, when sparsity is required, SIGMA outperforms all its competitors. In particular, \cref{subfig genl1} shows that SIGMA achieves a quadratic rate and is at least two times faster compared to the Newton and SubNewton methods. On the other hand, NewSamp, GD and SGD fail to reach the required tolerance within the time limit. Finally, we note that for smaller values of $p$ and/or larger input matrices $\w{A}$ the comparison is even more favorable for SIGMA.
\begin{figure}[!htb]\centering
	\begin{subfigure}{.5\textwidth}
		\centering
		\includegraphics[width=1\linewidth]{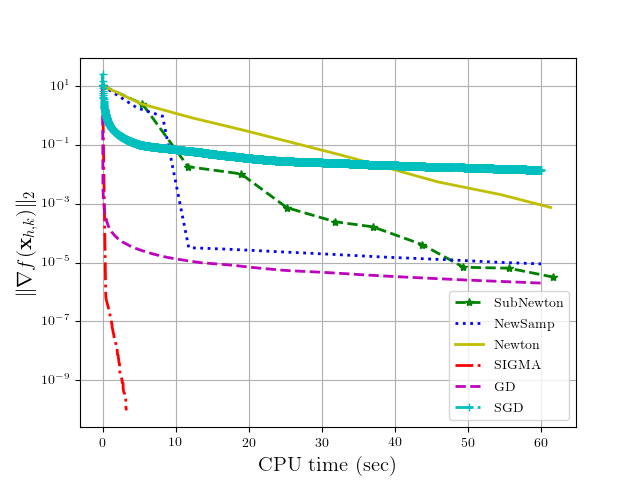}
		\caption{Leukemia}
		\label{subfig leu}
	\end{subfigure}
	\begin{subfigure}{.5\textwidth}
		\centering
		\includegraphics[width=1\linewidth]{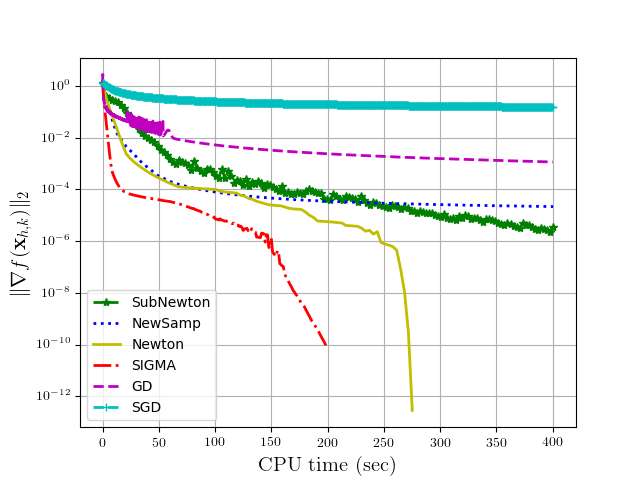}
		\caption{Gissette}
		\label{subfig gissette}
	\end{subfigure}
	\begin{subfigure}{.5\textwidth}
		\centering
		\includegraphics[width=1\linewidth]{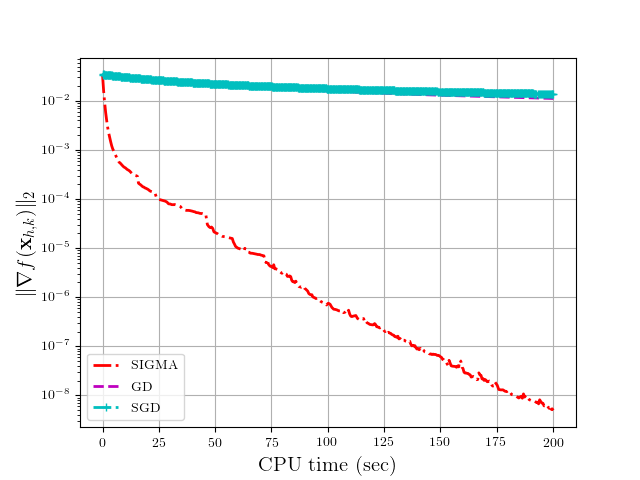}
		\caption{Real-sim}
		\label{subfig realisim}
	\end{subfigure}
	\caption{Performance of various optimization methods on different datasets for the logistic regression.}
	\label{fig logistic}
\end{figure}

\subsection{Logistic Regression} \label{subsec log regr}
In this set of experiments we report the performance of the optimization algorithms on the logistic model and three real datasets. In the first one (\cref{subfig leu}) we consider the Leukemia datasets for which $m \ll N$.  The coarse model dimensions for SIGMA is set to $n = 0.1N$. As for the Newsamp and SubNewton, $|S_m| = 0.1 m$ data points where used to form  the Hessian. Note that since here $m \ll N$, one should expect an eigenvalue gap located near $p = m$ and since $m$ is small SIGMA should enjoy a very fast convergence rate. This is illustrated in \cref{subfig leu} where, in particular, SIGMA converges in few seconds while its competitors are unable to get close to the true minimizer before the time exceeds. Note also that SIGMA obtains a convergence behavior similar to the one in \cref{subfig kappa} but it is sharper here since $p$ is much smaller than $N$. The only method that comes closer to our approach is the GD method which reaches a satisfactory tolerance very fast but then it slows down, and thus it struggles to approach the optimal point.

In the second experiment we consider the Gissete dataset with elastic-net regularization. For this example we set $n = N/2$ and $|S_m| = m/2$. The performance of the optimization algorithms is illustrated in \cref{subfig gissette} and observe that SIGMA outperforms all its competitors. The Newton method achieves a quadratic rate and thus reaches very high accuracy, but it is slower than SIGMA as the latter clearly achieves a super-linear rate and has cheaper iterates. The sub-sampled Newton methods reach a satisfactory tolerance within the time limit, however, note that very sparse solutions are only obtained in high accuracy. The GD methods on the other hand fail to even reach a sufficiently good solution before the time exceeds.

We end this set of experiments with the real-sim dataset over the $\ell_2$-regularized logistic regression. Since here both $m$ and $N$ are quite large, other than multilevel methods, only gradient-based methods can be employed to minimize the logistic model due to memory limitations. The comparison of the performance between SIGMA, GD and SGD is illustrated in \cref{subfig realisim}. Indisputably, \cref{subfig realisim} shows the efficiency of SIGMA against the first-order methods. In this example, SIGMA is able to return a very accurate solution while GD methods perform poorly, achieving a very slow linear convergence rate even from the start of the process.

\subsection{{Different Sampling Strategies}} \label{sec exps sample strategies}
{
In this section, we revisit the numerical experiments of the previous sections to compare SIGMA with the adaptive and mixed sampling strategies against conventional SIGMA which collects samples uniformly.
For the mixed strategy, we set $\tau = 0.5$ to all experiments. Besides parameter $\tau$, we consider the same set-up and tuning  (see sections above and \cref{subsec linear regr}). The results can be found in Figures \ref{fig extra logistic}, \ref{fig extra linear} and \ref{fig extra log and poisson}. Here, we report iterations instead  of CPU time as
we are interested in the effect of the sampling strategies on the convergence rate of SIGMA. 
}
\begin{figure}[!t]
	\begin{subfigure}{.5\textwidth}
		\centering
		\includegraphics[width=1\linewidth]{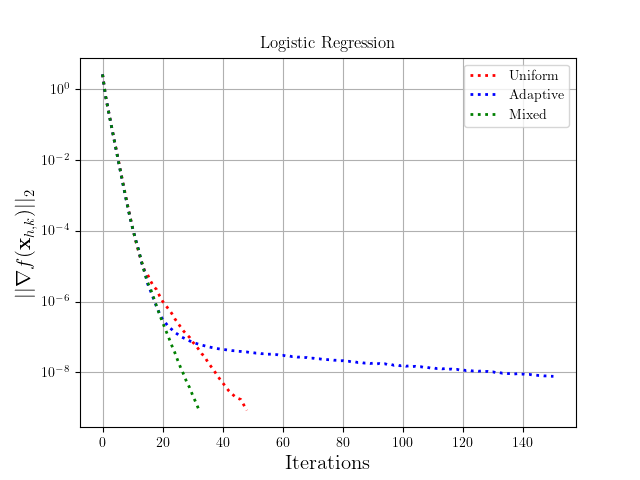}
		\caption{Gisette}
		\label{subfig extra gisette l2}
	\end{subfigure}
	\begin{subfigure}{.5\textwidth}
		\centering
		\includegraphics[width=1\linewidth]{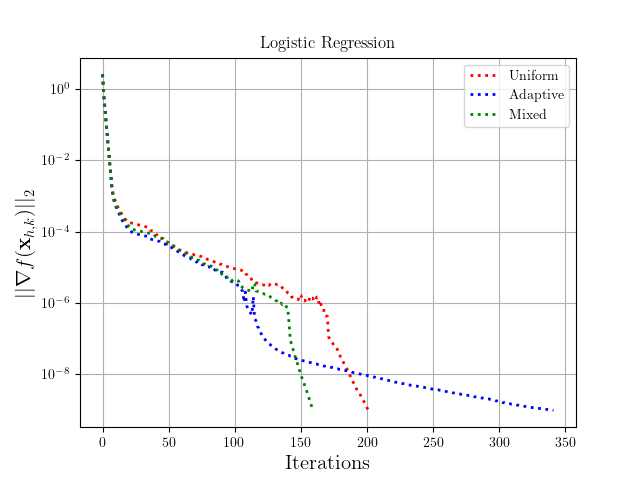}
		\caption{Gisette with $\ell_1$ regularization}
		\label{subfig extra gisette l1}
	\end{subfigure}
	\begin{subfigure}{.5\textwidth}
		\centering
		\includegraphics[width=1\linewidth]{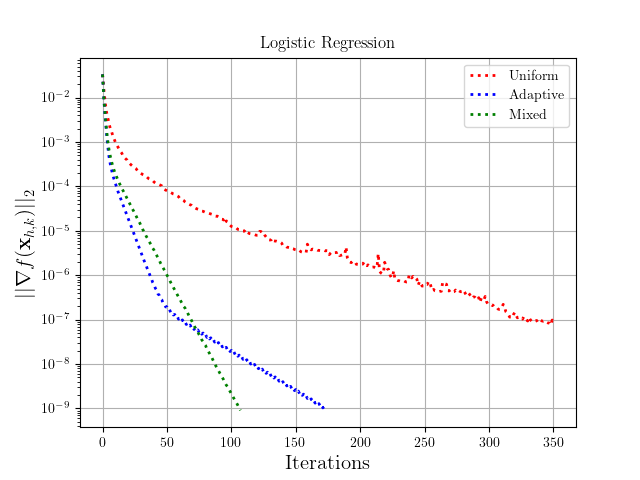}
		\caption{Real-sim}
		\label{subfig extra realisim}
	\end{subfigure}
	\caption{Comparisons between SIGMA with Uniform, Adaptive and mixed sampling strategies.}
	\label{fig extra logistic}
\end{figure}

\begin{figure}[!t]
	\begin{subfigure}{.5\textwidth}
		\centering
		\includegraphics[width=1\linewidth]{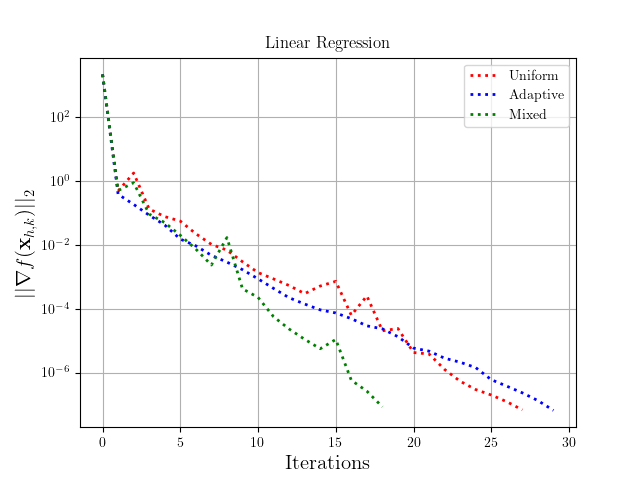}
		\caption{CtSlices}
		\label{subfig extra ctslices l2}
	\end{subfigure}
	\begin{subfigure}{.5\textwidth}
		\centering
		\includegraphics[width=1\linewidth]{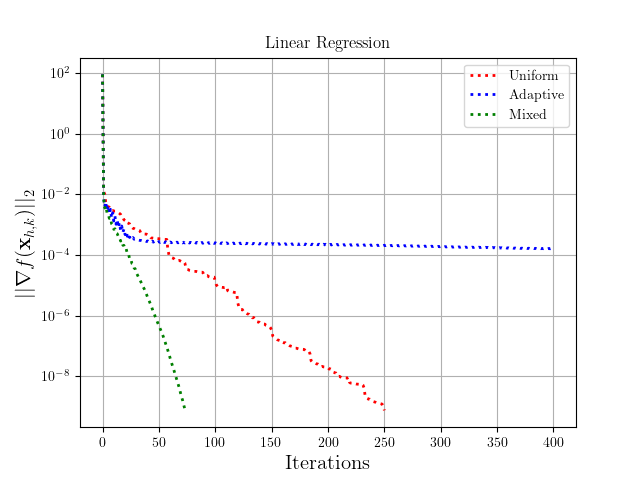}
		\caption{CMHS}
		\label{subfig extra cmhs}
	\end{subfigure}
	\begin{subfigure}{.5\textwidth}
		\centering
		\includegraphics[width=1\linewidth]{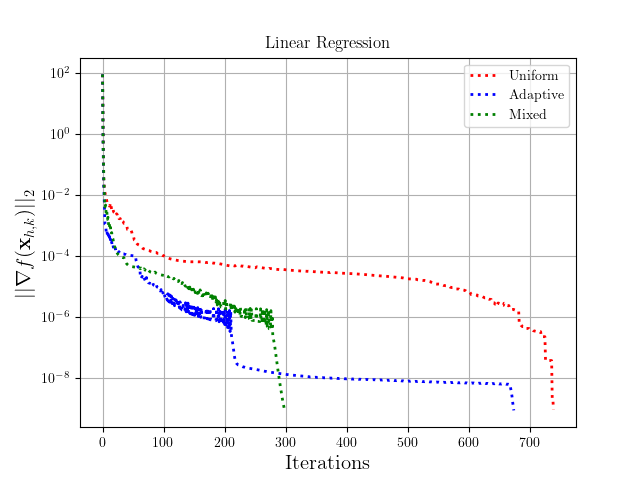}
		\caption{CMHS with $\ell_1$ regularization}
		\label{subfig extra cmhs l1}
	\end{subfigure}
	\caption{Comparisons between SIGMA with Uniform, Adaptive and mixed sampling strategies.}
	\label{fig extra linear}
\end{figure}

{
Observe that, in all the experiments, the adaptive and/or mixed strategies show substantial improvements in the convergence rate against of the conventional SIGMA. We notice that the adaptive strategy is particularly efficient in the early stages of the algorithm and it, in some examples, tends to slow down near the solution, especially when high accuracy is required. Therefore, when the goal is very accurate solutions, the results suggest that the uniform sampling may yield better convergence rates in the second-phase of SIGMA.
On the other hand, SIGMA with the mixed strategy significantly improves the convergence rate of the conventional SIGMA in all experiments, and both of its phases. Note that the mixed strategy was found to be at least $50\%$ faster and can be up to five times faster than SIGMA with uniform sampling (see \cref{subfig extra realisim}). Recall that the weight parameter $\tau$ is set to $0.5$ to account equally for adaptive and uniform sampling. However, as the experiments indicate, it would be more beneficial to set an adaptive (iteration-dependent) value of $\tau_k$ that starts from one (adaptive samples) and decays gradually towards zero (uniform samples). Therefore, even faster convergence rates will be expected for SIGMA. As a result, with the experiments presented in this section, we significantly improve the convergence behavior of multilevel or subspace method compared to the conventional methods that draw samples uniformly. In addition, SIGMA with the adaptive and mixed sampling strategies compares even more favorably to the state-of-the-art methods considered in previous sections.
}

\section{Conclusions and Perspectives} \label{sec conclusions}

\begin{figure}[t]
	\begin{subfigure}{.3\textwidth}
		\centering
		\includegraphics[width=1\linewidth]{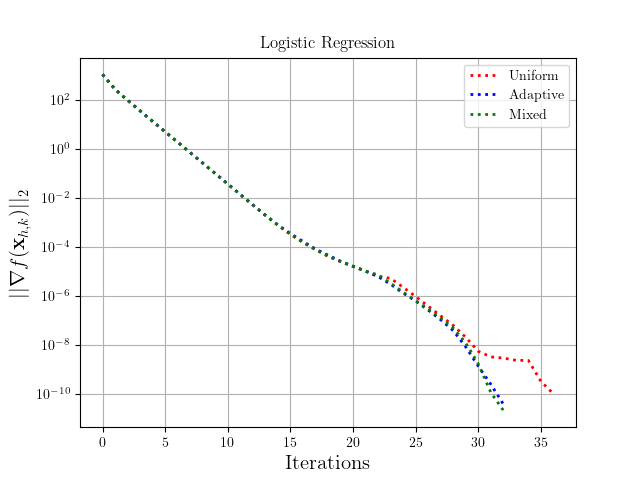}
		\caption{CtSlices}
		\label{subfig extra ctslices l2 logistic}
	\end{subfigure}
	\begin{subfigure}{.3\textwidth}
		\centering
		\includegraphics[width=1\linewidth]{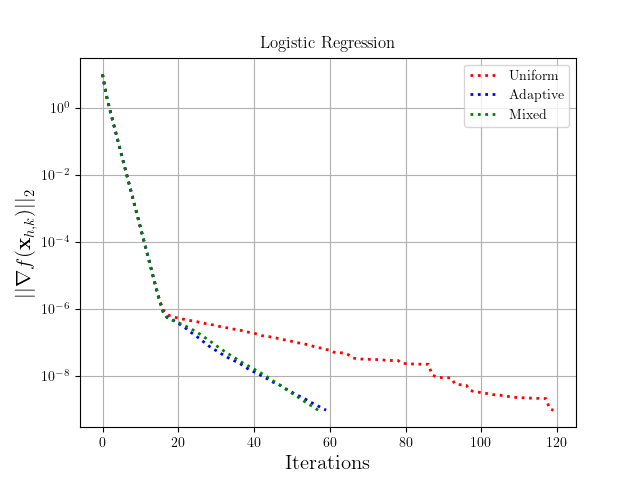}
		\caption{Leukemia}
		\label{subfig extra leukemia}
	\end{subfigure}
	\begin{subfigure}{.3\textwidth}
		\centering
		\includegraphics[width=1\linewidth]{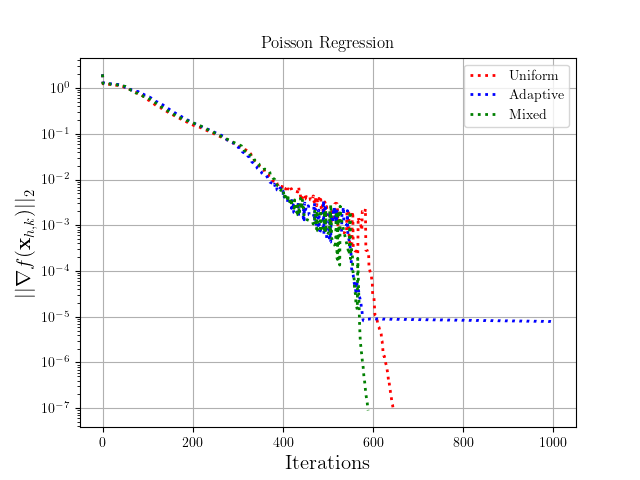}
		\caption{Synthetic with $\ell_1$ regularization}
		\label{subfig extra synthetic l1}
	\end{subfigure}
	\caption{Comparisons between SIGMA with Uniform, Adaptive and mixed sampling strategies.}
	\label{fig extra log and poisson}
\end{figure}

{We proposed SIGMA, a second-order variant of the Newton algorithm. We performed convergence analysis with the theory of self-concordant functions. We addressed two significant weaknesses of existing second-order methods for machine learning applications. In particular, the lack of global scale-invariant analysis and local super-linear convergence rates without restrictive assumptions. Our theory is general, in the sense that we do not assume specific problem structures. The theory allows also for a quadratic convergence rate. When the coarse direction is constructed randomly through the Nystr\"om method, the convergence analysis is provided in probability. Further, we introduce new adaptive sampling strategies to replace the conventional uniform sampling scheme. 
We show that the new strategies offer convergence results with higher probabilities than when considering the uniform sampling. In addition, we verify the theoretical results through numerical experiments. We illustrate that SIGMA can achieve both super-linear and quadratic convergence rates on realistic models. We, further, report substantial improvements in the convergence rate of SIGMA with the alternative strategies against the conventional SIGMA. As a future direction, we plan to extensively examine the practical behavior of SIGMA using variants of the new adaptive techniques. In particular, we plan to introduce adaptive and iteration-dependent $\tau_k$. We also consider replacing the $\|\cdot\|_1$ in the definition of adaptive sampling strategy with other norms. We anticipate that these modifications will offer better approximations of the Newton decrement. }

\begin{acknowledgements}
PP was partly supported by JPMorgan Chase $\&$ Co under a J.P. Morgan A.I. Research Faculty Award 2019/2021 and EPSRC Grant EP/W003317/1 
\end{acknowledgements}

All data supporting the findings of this study are available at the following URL: \url{https://archive.ics.uci.edu/datasets}.

\appendix  
\section{Appendix:  Bounds for Self-Concordant Functions}
In this section we prove some general results for self-concordant functions that will be required for the convergence analysis. 
\begin{lemma} \label{gama decrement properties}
For the approximate decrement in \cref{gama decrement} it holds that
\begin{enumerate}[label=(\roman*)]
 \item $\hat{\lambda}(\w{x}_{h,k})^2 = - \nabla f_h (\w{x}_{h,k})^T \hat{\w{d}}_{h,k},$
 \item $\hat{\lambda}(\w{x}_{h,k})^2 = \hat{\w{d}}_{h,k}^T \nabla^2 f_h(\w{x}_{h,k}) \hat{\w{d}}_{h,k} = \| \hat{\w{d}}_{h,k} \|_{\w{x}_{h,k}}^2,$
 \item $\hat{\lambda}(\w{x}_{h,k})^2 =  \hat{\w{d}}_{h,k}^T \nabla^2 f_h(\w{x}_{h,k}) \w{d}_{h,k},$
\end{enumerate}
where $\hat{\w{d}}_{h,k}$ is defined in \cref{galerkin d_h}, $\w{d}_{h,k}$ is the Newton direction and $\| \cdot \|_{\w{x}_{h,k}}$ in \cref{definition norm_x}. 
\end{lemma}
\begin{proof}
The results can be shown by direct replacement of the definitions of $\hat{\w{d}}_{h,k}$ and $\w{d}_{h,k}$ respectively. \qed
\end{proof}
Next, using the update rule in \cref{gama scheme} we derive bounds for the Hessian matrix.

\begin{lemma} \label{proposition f}
Let $f_h : \mathbb{R}^N \rightarrow \mathbb{R}$ satisfy \cref{ass self-conc}. If $ \hat{\lambda}(\w{x}_{h,k})< \frac{1}
{t_{h,k}}$, for the iterative scheme \cref{gama scheme} we have that
\begin{enumerate}[label=(\roman*)]
 \item $\nabla^2 f_h(\w{x}_{h,k+1}) \preceq \frac{1}{(1 - t_{h,k}\hat{\lambda}(\w{x}_{h,k}))^2} \nabla^2 f_h(\w{x}_{h,k}),$ 
 \label{proposition f case i}
 \item $[\nabla^2 f_h(\w{x}_{h,k+1})]^{-1} \preceq  \frac{1}{( 1 - t_{h,k}\hat{\lambda}(\w{x}_{h,k}))^2} 
 [\nabla^2 f_h(\w{x}_{h,k})]^{-1}$. \label{proposition f case ii}
\end{enumerate}
\end{lemma}
\begin{proof}
Consider the case \textit{\ref{proposition f case i}}. From the upper bound in \cref{relation sandwitz} that arise for self-concordant functions we have that,
\begin{align*}
 \nabla^2 f_h(\w{x}_{h,k+1}) & \preceq \frac{1}{(1 - t_{h,k} \| \hat{\w{d}}_{h,k} \|_{\w{x}_{h,k}})^2}\nabla^2 f_h(\w{x}_{h,k}) \\
			      & = \frac{1}{(1 - t_{h,k}\hat{\lambda}(\w{x}_{h,k}))^2} \nabla^2 f_h(\w{x}_{h,k}).
\end{align*}
which holds for $\hat{\lambda}(\w{x}_{h,k})< 1/t_{h,k}$ as claimed. As for the case \textit{\ref{proposition f case ii}}, we make use of the lower bound in \cref{relation sandwitz}, and thus, for $\hat{\lambda}(\w{x}_{h,k})< 1/t_{h,k}$,
\begin{equation*}
 \nabla^2 f_h(\w{x}_{h,k+1}) \succeq (1 - t_{h,k}\hat{\lambda}(\w{x}_{h,k}))^2 \nabla^2 f_h(\w{x}_{h,k}).
\end{equation*}
Since, further, $f_h$ is strictly convex we take,  
\begin{equation*}
 [\nabla^2 f_h(\w{x}_{h,k+1})]^{-1} \preceq  \frac{1}{(1 - t_{h,k}\hat{\lambda}(\w{x}_{h,k}))^2} [\nabla^2 f_h(\w{x}_{h,k})]^{-1},
\end{equation*}
which concludes the proof. \qed

\end{proof}

Similarly, we can obtain analogous bounds for the reduced Hessian matrix.

\begin{lemma} \label{proposition Q}
Let $f_h : \mathbb{R}^N \rightarrow \mathbb{R}$ satisfy \cref{ass self-conc}. If $ \hat{\lambda}(\w{x}_{h,k})< \frac{1}{t_{h,k}}$, for the iterative scheme \cref{gama scheme} we have that
\begin{enumerate}[label=(\roman*)]
 \item $\w{Q}_H(\w{x}_{h,k+1}) \preceq \frac{1}{(1 - t_{h,k}\hat{\lambda}(\w{x}_{h,k}))^2} \w{Q}_H(\w{x}_{h,k})$, \label{proposition Q case i}
 \item $[\w{Q}_H(\w{x}_{h,k+1})]^{-1} \preceq  \frac{1}{(1 - t_{h,k}\hat{\lambda}(\w{x}_{h,k}))^2} 
 [\w{Q}_H(\w{x}_{h,k})]^{-1}$. \label{proposition Q case ii}
\end{enumerate}
\end{lemma}
\begin{proof}
We already know that,
\begin{equation*}
\nabla^2 f_h(\w{x}_{h,k+1}) \preceq \frac{1}{(1 - t_{h,k}\hat{\lambda}(\w{x}_{h,k}))^2} \nabla^2 f_h(\w{x}_{h,k}).
\end{equation*}
By strict convexity and \cref{assumption P} we see that,
\begin{equation*}
 \w{Q}_H(\w{x}_{h,k+1}) \preceq \frac{1}{(1 - t_{h,k}\hat{\lambda}(\w{x}_{h,k}))^2} \w{Q}_H(\w{x}_{h,k}),
\end{equation*}
which is exactly the bound in case \textit{\ref{proposition Q case i}}. Next, recall that,
\begin{equation*}
 \nabla^2 f_h(\w{x}_{h,k+1}) \succeq (1 - t_{h,k}\hat{\lambda}(\w{x}_{h,k}))^2 \nabla^2 f_h(\w{x}_{h,k}),
\end{equation*}
and, again, by strict convexity and \cref{assumption P} we have that,
\begin{equation*}
 \w{Q}_H(\w{x}_{h,k+1}) \succeq (1 - t_{h,k}\hat{\lambda}(\w{x}_{h,k}))^2 \w{Q}_H(\w{x}_{h,k}).
\end{equation*}
In addition, using the above relation and \cref{proposition Q_H pd} we can obtain the bound in case \textit{\ref{proposition Q case ii}}. Finally, all the bounds hold for $\hat{\lambda}(\w{x}_{h,k})< 1/t_{h,k}$ which concludes the proof. \qed
\end{proof}

In our analysis, we will further make use of two bounds that hold for self-concordant functions.  We only state the results, the the proofs can be found in \cite{MR2061575}.
\begin{lemma}{\cite{MR2061575}} \label{proposition phi}
Let $\phi : \mathbb{R} \rightarrow \mathbb{R}$ be a strictly convex self-concordant function. Then,
\begin{enumerate}[label=(\roman*)]
 \item $ \phi(t) \leq \phi(0) + t\phi'(0)  - t \phi''(0)^{1/2} - \log(1-t\phi''(0)^{1/2}), \ \ \
 t \leq \phi''(0)^{-1/2}, $ \label{proposition phi case i}
 \item $ \phi(t) \geq \phi(0) + t\phi'(0)  + t \phi''(0)^{1/2} - \log(1+t\phi''(0)^{1/2}), \ \ \ t \geq 0 $. \label{proposition phi case ii}
\end{enumerate}
\end{lemma}

\begin{proof}
Both inequalities can be proved using relation (\ref{relation univariate sandwitz}), for details see \cite{MR2061575}.\qed
\end{proof}

\section{Appendix: Proofs of Theorems and Lemmas}

\begin{proof}[Proof of \cref{prop select mu}]
Note that $\mu \in (0, \min \{ 1, \frac{\nu}{\lambda(\w{x}_{h,0})} \} )$ implies $\nu > \mu \lambda(\w{x}_{h,0}) $. Thus, $\hat{\lambda}(\w{x}_{h,k}) > \nu > \mu \lambda(\w{x}_{h,0}) \geq \cdots \geq \mu \lambda(\w{x}_{h,k})$, as required. \qed
\end{proof}

\begin{proof}[Proof of \cref{lemma upper bound lamda_hat}]
By the definition of $\hat{\lambda}(\w{x}_{h,k})$ in \cref{gama decrement} we have that,
\begin{align*}
     \hat{\lambda}(\w{x}_{h,k}) & = \left[  \nabla f_h (\w{x}_{h,k})^T \w{P} [\w{R} \nabla^2 f_h(\w{x}_{h,k}) \w{P}]^{-1} \w{R} \nabla f_h (\w{x}_{h,k}) \right]^{1/2}.\\
     & = \left[  \nabla f_h (\w{x}_{h,k})^T \nabla^2 f_h(\w{x}_{h,k})^{-\frac{1}{2}} \Pi \nabla^2f_h(\w{x}_{h,k})^{-\frac{1}{2}} \nabla f_h (\w{x}_{h,k}) \right]^{1/2}, \\
\end{align*}
where $ \Pi =  (\nabla^2 f_h(\w{x}_{h,k})^\frac{1}{2} \w{P}) [(\nabla^2 f_h(\w{x}_{h,k})^\frac{1}{2} \w{P})^T (\nabla^2 f_h(\w{x}_{h,k})^\frac{1}{2} \w{P})]^{-1} (\nabla^2 f_h(\w{x}_{h,k})^{\frac{1}{2}} \w{P})^T $ is the orthogonal projection onto the $\operatorname{span}\{ [\nabla^2 f_h(\w{x}_{h,k})]^{\frac{1}{2}} \w{P}\}$. By the idepotency of the orthogonal projection we have that,
\begin{align*}
     \hat{\lambda}(\w{x}_{h,k}) & = \left\langle \Pi \nabla^2 f_h(\w{x}_{h,k})^{-\frac{1}{2}} \nabla f_h (\w{x}_{h,k}), \Pi \nabla^2 f_h(\w{x}_{h,k})^{-\frac{1}{2}} \nabla f_h (\w{x}_{h,k}) \right\rangle^{1/2} \\
     & \leq \left\| \Pi \right\|  \left\| \nabla^2 f_h(\w{x}_{h,k})^{-\frac{1}{2}} \nabla f_h (\w{x}_{h,k}) \right\|.
\end{align*}
Since $\sigma_1(\Pi) = 1$ we take $\hat{\lambda}(\w{x}_{h,k}) \leq \lambda(\w{x}_{h,k})$ as required. \qed
\end{proof}

Let $\phi(t_{h,k}) := f_h(\w{x}_{h,k} + t_{h,k} \w{\hat{d}}_{h,k})$. It is easy to show that,
\begin{equation*}
 \phi'(0) = - \hat{\lambda}(\w{x}_k)^2 \quad \text{and} \quad \phi''(0) =  \hat{\lambda}(\w{x}_k)^2,
\end{equation*}
where by definition $\phi : \mathbb{R} \rightarrow \mathbb{R} $ satisfies \cref{ass self-conc}.

\begin{proof}[Proof of \cref{lemma411}]
By \cref{proposition phi} we have that,
\begin{align*}
 \phi(t_{h,k}) & \leq \phi(0) + t_{h,k} \phi'(0) - t_{h,k} \phi''(0)^{1/2}- \log \left(1 - t_{h,k} \phi''(0)^{1/2}\right) \\
	 &  =  \underbrace{ \phi(0) - t_{h,k} \hat{\lambda}(\w{x}_{h,k})^2 - t_{h,k} \hat{\lambda}(\w{x}_{h,k}) - \log \left(1 - t_{h,k} 
	 \hat{\lambda}(\w{x}_{h,k})\right)}_{h(t_{h,k})},
\end{align*}
which is valid for $t_{h,k} < \frac{1}{\hat{\lambda}(\w{x}_{h,k})}$. Note that $h(t_{h,k})$ is minimized at $t^*_h = 
\frac{1}{1 + \hat{\lambda}(\w{x}_{h,k})}$ and thus, 
\begin{align*}
 \phi(t^*_h) &  \leq \phi(0) -  \frac{\hat{\lambda}(\w{x}_{h,k})^2}{1 + \hat{\lambda}(\w{x}_{h,k})} - \frac{\hat{\lambda}(\w{x}_{h,k})}{1 + 
 \hat{\lambda}(\w{x}_{h,k})} 
 - \log \left(1 - \frac{\hat{\lambda}(\w{x}_{h,k})^2}{1 + \hat{\lambda}(\w{x}_{h,k})} \right) \\
	   &= \phi(0) - \hat{\lambda}(\w{x}_{h,k}) + \log \left( 1 + \hat{\lambda}(\w{x}_{h,k}) \right).
\end{align*}
Using the inequality,
\begin{equation*}
 -x + \log(1+x) \leq - \frac{x^2}{2(1+x)},
\end{equation*}
for any $x>0$, we obtain the following upper bound for $\phi(t_h^*)$,
\begin{align*}
 \phi(t^*_h) &\leq \phi(0) - \frac{\hat{\lambda}(\w{x}_{h,k})^2}{2(1 + \hat{\lambda}(\w{x}_{h,k}))} \\
	   & \leq \phi(0) - \alpha t^*_h  \hat{\lambda}(\w{x}_{h,k})^2 \\
	   & = \phi(0) + \alpha t^*_h \nabla f^T_{h}(\w{x}_{h,k}) \w{\hat{d}}_{h,k}.
\end{align*}
Thus $t^*_h$ satisfies the back-tracking line search exit condition which means that it will always return a step size $t_{h,k} > \beta / (1 + \hat{\lambda}(\w{x}_{h,k})) $. Therefore,
\begin{equation*}
f_h(\w{x}_{h,k} + t_{h,k} \w{\hat{d}}_{h,k}) - f_h(\w{x}_{h,k}) \leq - \alpha \beta \frac{\hat{\lambda}(\w{x}_{h,k})^2}{1 + \hat{\lambda}(\w{x}_{h,k})}.
\end{equation*}
Additionally, since $\hat{\lambda}(\w{x}_{h,k}) > \eta$ and using the fact that the function $x \rightarrow \frac{x^2}{1 + x} $ is monotone 
increasing for any $x>0$, we have that,
\begin{equation*}
 f_h(\w{x}_{h,k} + t_{h,k} \w{\hat{d}}_{h,k}) - f_h(\w{x}_{h,k}) \leq - \alpha \beta \frac{\eta^2}{1 + \eta}.
\end{equation*}
which concludes the proof by setting $\gamma = \alpha \beta \eta^2 / (1 + \eta)$. \qed
\end{proof}

\begin{proof}[Proof of \cref{lemma412}]
Using the second inequality in \cref{proposition phi} we obtain the following bound, 
\begin{align*}
 \phi(t_{h,k}) & \geq \phi(0) + t_{h,k} \phi'(0) + t_{h,k} \phi''(0)^{1/2} - \log(1 + t_{h,k} \phi'(0)^{1/2}) \\
	       & = \underbrace{\phi(0) - t_{h,k} \hat{\lambda}(\w{x}_{h,k})^2 + t_{h,k} \hat{\lambda}(\w{x}_{h,k}) - \log(1 + t_{h,k} \hat{\lambda}(\w{x}_{h,k}))}_{g(t_{h,k})},  
\end{align*}
which is true for any $t_{h,k} \geq 0$. Moreover, the function $g(t_{h,k})$ is minimized at $t^*_h = 1/(1 - \hat{\lambda}(\w{x}_{h,k}))$ and thus,
\begin{align*}
 \underset{t_{h,k}\geq0}{\operatorname{inf}} \{ \phi(t_{h,k}) \} & \geq \phi(0) - \frac{\hat{\lambda}(\w{x}_{h,k})^2}{1 - \hat{\lambda}(\w{x}_{h,k})} + 
 \frac{\hat{\lambda}(\w{x}_{h,k})}{1 - \hat{\lambda}(\w{x}_{h,k})} - \log(1 + \frac{\hat{\lambda}(\w{x}_{h,k})}{1 - \hat{\lambda}(\w{x}_{h,k})}) \\
 & = \phi(0) + \hat{\lambda}(\w{x}_{h,k}) + \log(1 - \hat{\lambda}(\w{x}_{h,k})),
\end{align*}
which is valid since, by assumption, $\hat{\lambda}(\w{x}_{h,k}) < 1$. Then, $ f_h(\w{x}_h^*) \geq f_h(\w{x}_{h,k}) - \omega_*(\hat{\lambda}(\w{x}_{h,k}))$, and thus the upper bound is proved.

Furthermore, in view of \cref{lemma411}, we have that,
\begin{align*}
\underset{t_{h,k}\geq0}{\operatorname{inf}} \{ \phi(t_{h,k}) \} & \leq    \underset{t_{h,k}\geq0}{\operatorname{inf}} \Big\{ \phi(0) - t_{h,k} \hat{\lambda}(\w{x}_{h,k})^2 - t_{h,k} \hat{\lambda}(\w{x}_{h,k}) - \log \left(1 - t_{h,k} \hat{\lambda}(\w{x}_{h,k})\right) \Big\} \\
 &= \phi(0) - \hat{\lambda}(\w{x}_{h,k}) + \log \left( 1 + \hat{\lambda}(\w{x}_{h,k}) \right).
\end{align*}
Then, $ f_h(\w{x}_h^*) \leq f_h(\w{x}_{h,k}) - \omega(\hat{\lambda}(\w{x}_{h,k}))$, and thus the lower bound is proved which concludes the proof of the lemma. \qed
\end{proof}

\begin{proof}[Proof of \cref{lemma suboptimality gap}]
Combining the upper bound in \cref{lemma412}, the result in \cref{lemma upper bound lamda_hat} and since $\omega_*(x)$ in \cref{omegas} is monotone increasing we have that
\begin{equation*}
     f_h(\w{x}_{h,k}) - f_h(\w{x}_h^*) \leq \omega_*(\lambda(\w{x}_{h,k})),
\end{equation*}
which holds for $\lambda(\w{x}_{h,k}) < 1$. Further, since $ \omega_*(x) \leq x^2, \ x \in [0, 0.68] $ the claim follows if $\lambda(\w{x}_{h,k}) \leq 0.68$. \qed
\end{proof}

\begin{proof}[Proof of \cref{lemma415}]
From \cref{lemma411} recall that,
\begin{equation*}
 \phi(t_{h,k}) \leq \phi(0) - t_{h,k} \hat{\lambda}(\w{x}_{h,k})^2 - t_{h,k} \hat{\lambda}(\w{x}_{h,k}) - \log \left(1 - t_{h,k}
 \hat{\lambda}(\w{x}_{h,k})\right)
\end{equation*}
which is valid for  $\hat{\lambda}(\w{x}_{h,k}) < 1 / t_{h,k}$. Setting $t_{h,k}=1$ we have that, 
\begin{equation*}
  \phi(1) \leq \phi(0) -  \hat{\lambda}(\w{x}_{h,k})^2 - \hat{\lambda}(\w{x}_{h,k}) - \log \left(1 - \hat{\lambda}(\w{x}_{h,k})\right)
\end{equation*}
with $\hat{\lambda}(\w{x}_{h,k}) < 1 $. Further, as in \cite{MR2061575}, making use of the inequality
\begin{equation*}
 -x - \log(1-x) \leq \frac{1}{2} x^2 + x^3, \quad x \in [0, 0.81]
\end{equation*}
we get,
\begin{align*}
 \phi(1) & \leq \phi(0) - \frac{1}{2} \hat{\lambda}(\w{x}_{h,k})^2 + \hat{\lambda}(\w{x}_{h,k})^3 
	 = \phi(0) - \frac{1}{2} \left( 1 - 2 \hat{\lambda}(\w{x}_{h,k}) \right) \hat{\lambda}(\w{x}_{h,k})^2
\end{align*}
which holds for $\hat{\lambda}(\w{x}_{h,k}) \leq 0.81$. Setting $\alpha \leq \frac{1}{2} (1 - 2\hat{\lambda}(\w{x}_{h,k}))$ we obtain, 
\begin{equation*} \label{unit step ineq}
 f_h(\w{x}_{h,k} + \w{\hat{d}}_{h,k}) \leq f_h(\w{x}_{h,k}) - \alpha \hat{\lambda}(\w{x}_{h,k})^2,
\end{equation*}
which satisfies the backtracking line search condition for $t_{h,k} = 1$ and $\hat{\lambda}(\w{x}_{h,k}) \leq \frac{1}{2} (1 - 2\alpha)$, which concludes the proof. \qed
\end{proof}

\begin{proof}[Proof of \cref{thm422}]
By the definition of the approximate decrement we have that,
\begin{equation*}
 \hat{\lambda}(\w{x}_{h,k}) =  \sqrt{\nabla f_h (\w{x}_{h,k})^T \w{P} [\w{Q}_{H}(\w{x}_{h,k}) ]^{-1} \w{R} \nabla f_h (\w{x}_{h,k})},
\end{equation*}
In addition, from \cref{proposition Q}, and since $t_{h,k} = 1$, we get,
\begin{equation*}
  [\w{Q}_{H}(\w{x}_{h,k+1})]^{-1} \preceq  \frac{1}{(1-\hat{\lambda}(\w{x}_{h,k}))^2}  [\w{Q}_{H}(\w{x}_{h,k})]^{-1},
\end{equation*}
which holds since, by assumption, $\hat{\lambda}(\w{x}_{h,k})<1$. Using this relation into the definition of $\hat{\lambda}(\w{x}_{h,k+1})$ above, 
we have that,
\begin{align*}
\hat{\lambda}(\w{x}_{h,k+1})  & \leq \frac{1}{1-\hat{\lambda}(\w{x}_{h,k})} \left\| [\w{Q}_{H}(\w{x}_{h,k})]^{-1/2} \w{R} \nabla f_h (\w{x}_{h,k+1}) \right\|_2.
\end{align*}
Further, observe that $\nabla f_h (\w{x}_{h,k+1}) = \int_0^1 \nabla^2 f_h ({\w{x}_{h,k}} + y \w{\hat{d}}_{h,k} ) \w{\hat{d}}_{h,k} \ dy + 
\nabla f_h (\w{x}_{h,k}) $ and thus,
\begin{equation} \label{a1 a2 ineq}
 \hat{\lambda}(\w{x}_{h,k+1}) \leq \frac{1}{1-\hat{\lambda}(\w{x}_{h,k})} \left\| \w{A}_1 + \w{A}_2 \right\|_2,
\end{equation}
where we denote $\w{A}_1 := [\w{Q}_{H}(\w{x}_{h,k})]^{-1/2} \w{R} \int_0^1 \nabla^2 f_h ({\w{x}_{h,k}} + y \w{\hat{d}}_{h,k} ) \w{\hat{d}}_{h,k} \ dy$
and \\ $\w{A}_2 := [\w{Q}_{H}(\w{x}_{h,k})]^{-1/2} \w{R} \nabla f_h (\w{x}_{h,k})$.
By the definitions of the coarse step in \cref{galerkin d_h} and \cref{galerkin d_H}, we have that $\w{\hat{d}}_{h,k} = \w{P} \w{\hat{d}}_{H,k}$, 
and thus, using simple algebra, $\w{A}_1$ and $\w{A}_2$ become,
\begin{align*}
 \w{A}_1 & = [\w{Q}_{H}(\w{x}_{h,k})]^{-1/2} \int_0^1 \w{Q}_{H}(\w{x}_{h,k}  + y \w{\hat{d}}_{h,k} ) \w{\hat{d}}_{H,k} \ dy \\
     & = \int_0^1 [\w{Q}_{H}(\w{x}_{h,k})]^{-1/2}  \w{Q}_{H}(\w{x}_{h,k}  + y \w{\hat{d}}_{h,k} ) [\w{Q}_{H}(\w{x}_{h,k})]^{-1/2} \
     dy \  [\w{Q}_{H}(\w{x}_{h,k})]^{1/2} \w{\hat{d}}_{H,k}, 
\end{align*}
and,
\begin{align*}
 \w{A_2} & = [\w{Q}_{H}(\w{x}_{h,k})]^{1/2} [\w{Q}_{H}(\w{x}_{h,k})]^{-1} \w{R} \nabla f_h (\w{x}_{h,k}) \\
	 & = - [\w{Q}_{H}(\w{x}_{h,k})]^{1/2} \w{\hat{d}}_{H,k},
\end{align*}
respectively. Then,
\begin{align*}
    & \w{A_1} + \w{A_2} = \\
    & \int_0^1 \left( [\w{Q}_{H}(\w{x}_{h,k})]^{-1/2}  \w{Q}_{H}(\w{x}_{h,k}  + y \w{\hat{d}}_{h,k} ) [\w{Q}_{H}(\w{x}_{h,k})]^{-1/2} - \w{I} \right)
     dy \  [\w{Q}_{H}(\w{x}_{h,k})]^{1/2} \w{\hat{d}}_{H,k}
\end{align*}
From \cref{proposition Q}\textit{\ref{proposition Q case i}}, we take $\w{Q}_{H}(\w{x}_{h,k}  +
y \w{\hat{d}}_{h,k} ) \preceq  \frac{1}{(1-y \hat{\lambda}(\w{x}_{h,k}))^2} \w{Q}_{H}(\w{x}_{h,k})$, which is valid since $y \hat{\lambda}(\w{x}_{h,k}) < 1$, and so \cref{a1 a2 ineq} can be bounded as follows,
\begin{align*}
\hat{\lambda}(\w{x}_{h,k+1}) 
			     & \leq \frac{1}{1-\hat{\lambda}(\w{x}_{h,k})} \left\| \int_0^1 \left( \frac{1}{(1-y \hat{\lambda} (\w{x}_{h,k}))^2} 
			      - 1 \right) \w{I}_{n \times n} dy \ [\w{Q}_{H}(\w{x}_{h,k})]^{1/2} \w{\hat{d}}_{H,k}  \right\|_2 \\
			     & \leq \frac{1}{1-\hat{\lambda}(\w{x}_{h,k})} \left\| \int_0^1 \left( \frac{1}{(1-y \hat{\lambda}(\w{x}_{h,k}))^2} - 1 \right) 
			     dy \right\|_2 \left\| [\w{Q}_{H}(\w{x}_{h,k})]^{1/2} \w{\hat{d}}_{H,k}  \right\|_2, 
\end{align*}
where $\w{I}_{n \times n}$ denotes the $n \times n$ identity matrix. Note that, 
\begin{equation*}
 \int_0^1 \left( \frac{1}{(1-y \hat{\lambda}(\w{x}_{h,k}))^2} - 1 \right) dy = \frac{\hat{\lambda}(\w{x}_{h,k}))}{(1 - \hat{\lambda}(\w{x}_{h,k}))},
\end{equation*}
and also that,
\begin{align*}
 \left\| [\w{Q}_{H}(\w{x}_{h,k})]^{1/2} \w{\hat{d}}_{H,k}  \right\|_2 &= \left( \left( \w{P} \w{\hat{d}}_{H,k} \right)^T \nabla^2 f_h(\w{x}_{h,k}) 
 \w{P} \w{\hat{d}}_{H,k}\right)^{\frac{1}{2}} \\ 
    &= \left( \w{\hat{d}}_{h,k}^T \nabla^2 f_h(\w{x}_{h,k}) \w{\hat{d}}_{h,k} \right)^{\frac{1}{2}} \\
    &= \hat{\lambda}(\w{x}_{h,k}),
\end{align*}
which concludes the proof of the theorem by directly replacing both equalities into the last inequality of $\hat{\lambda}(\w{x}_{h,k+1})$. \qed
\end{proof}

\begin{proof}[Proof of \cref{thm quadratic coarse through fine}]
Applying the proof of \cref{thm422} with the Newton direction instead of the coarse direction, we can show that
 \begin{equation*}
  \hat{\lambda}(\w{x}_{h,k+1}) \leq  \frac{\lambda(\w{x}_{h,k})}{(1 - \lambda(\w{x}_{h,k}))^2} \hat{\lambda}(\w{x}_{h,k}).
 \end{equation*}        
  Since the fine step is taken, we have that either $\hat{\lambda}(x_k) = 0$ or  $\hat{\lambda}(x_k) \leq \mu \lambda(x_k)$. It makes sense to measure the reduction of $\hat{\lambda}(x_k)$ only in the second case. Assume that $\hat{\lambda}(x_k) > \nu$ and $\hat{\lambda}(x_k) \leq \mu \lambda(x_k)$. Then we immediately take  $\nu < \hat{\lambda}(\w{x}_{h,k}) \leq \mu \lambda(\w{x}_{h,k}) \leq \lambda(\w{x}_{h,k}) < 1 $. Therefore, there exists a sufficiently small and non-negative real number, say $\mu_{1,k}$, such that $\nu< \mu_{1,k} \lambda(\w{x}_{h,k}) \leq \hat{\lambda}(\w{x}_{h,k}) \leq \mu \lambda(\w{x}_{h,k})$, and thus $\frac{\nu}{\lambda(\w{x}_{h,k})} < \mu_{1,k} \leq \mu $. By the monotonicity of $\lambda(\w{x}_{h,k})$ we have that $\frac{\nu}{\lambda(\w{x}_{h,0})} < \mu_{1,k} \leq \mu $. Putting this all together we obtain
  \begin{equation*}
      \hat{\lambda}(\w{x}_{h,k+1}) \leq  \frac{\hat{\lambda}(\w{x}_{h,k})^2}{\mu_{1,k}(1 - \lambda(\w{x}_{h,k}))^2} \leq \frac{\lambda(\w{x}_{h,0})}{\nu (1 - \lambda(\w{x}_{h,k}))^2} \hat{\lambda}(\w{x}_{h,k})^2
  \end{equation*}  
  as claimed. \qed
\end{proof}

\begin{proof}[Proof of \cref{lemma norm equality lambda's d's}]
We have that
\begin{align*}
    &\left\| \ [\nabla^2 f_h ({\w{x}_{h,k}})]^{1/2} \left(\w{\hat{d}}_{h,k} - \w{d}_{h,k} \right) \right\|_2  = \\
     &  \sqrt{ \| \w{\hat{d}}_{h,k} \|_{\w{x}_{h,k}}^2  +
			   \| \w{d}_{h,k}\|_{\w{x}_{h,k}}^2 -  2\w{\hat{d}}_{h,k}^T \nabla^2 f_h ({\w{x}_{h,k}}) \w{d}_{h,k}},
\end{align*}
where $\| \cdot \|_{\w{x}_{h,k}}$ is defined in \cref{definition norm_x}. Using now the results from \cref{gama decrement properties} and the definition of the Newton decrement in \cref{newton decrement}, the claim follows. \qed
\end{proof}

\begin{proof}[Proof of \cref{lemma417}]
By the definition of the Newton decrement we have that,
\begin{equation*}
 \lambda(\w{x}_{h,k+1}) = \left[ \nabla f_h(\w{x}_{h,k+1})^T [ \nabla^2 f_h(\w{x}_{h,k+1})]^{-1} \nabla f_h(\w{x}_{h,k+1})  \right]^{1/2}.
\end{equation*}
Combining $t_{h,k} = 1$ and \cref{proposition f} we take, 
\begin{equation} \label{ineqlambda}
 \lambda(\w{x}_{h,k+1}) \leq \frac{1}{1-\hat{\lambda}(\w{x}_{h,k})} \left\| [\nabla^2 f_h ({\w{x}_{h,k}})]^{-1/2} \nabla f_h(\w{x}_{h,k+1}) \right\|_2.
\end{equation}
Denote $\w{Z} := [\nabla^2 f_h ({\w{x}_{h,k}})]^{-1/2} \nabla f_h(\w{x}_{h,k+1})$. Using the fact that 
\begin{equation*}
\nabla f_h(\w{x}_{h,k+1}) = \int_0^1 \nabla^2 f_h ({\w{x}_{h,k}} + y \w{\hat{d}}_{h,k}) \w{\hat{d}}_{h,k} \ dy + \nabla f_h(\w{x}_{h,k})
\end{equation*}
we see that,
\begin{align*}
 \w{Z} & = [\nabla^2 f_h ({\w{x}_{h,k}})]^{-1/2} \left( \int_0^1 \nabla^2 f_h ({\w{x}_{h,k}} + y \w{\hat{d}}_{h,k}) \w{\hat{d}}_{h,k} \ dy 
	       + \nabla f_h(\w{x}_{h,k})  \right) \\
       & =  \underbrace{\int_0^1 [\nabla^2 f_h ({\w{x}_{h,k}})]^{-1/2} \nabla^2 f_h ({\w{x}_{h,k}} +y\w{\hat{d}}_{h,k})
       [\nabla^2 f_h ({\w{x}_{h,k}})]^{-1/2} \ dy}_\w{T} \ [\nabla^2 f_h ({\w{x}_{h,k}})]^{1/2} \w{\hat{d}}_{h,k} \\ 
       & \qquad \qquad \qquad \qquad \qquad \qquad \qquad \qquad \qquad \qquad \qquad \qquad \qquad   -    [\nabla^2 f_h ({\w{x}_{h,k}})]^{1/2} \w{d}_{h,k}, \\
\end{align*}
where $\w{d}_{h,k}$ is the Newton direction. Next, adding and subtracting $\w{T}[\nabla^2 f_h ({\w{x}_{h,k}})]^{1/2} \w{d}_{h,k}$ we have that,
\begin{equation*}
    \w{Z} = \w{T}[\nabla^2 f_h ({\w{x}_{h,k}})]^{1/2} (\w{\hat{d}}_{h,k} - \w{d}_{h,k}) + (\w{T} - \w{I}_{N \times N})[\nabla^2 f_h ({\w{x}_{h,k}})]^{1/2} \w{d}_{h,k},
\end{equation*}
and thus,
\begin{equation} \label{ineq Zs}
    \left\| \w{Z} \right\| \leq \left \|\underbrace{\w{T}[\nabla^2 f_h ({\w{x}_{h,k}})]^{1/2} (\w{\hat{d}}_{h,k} - \w{d}_{h,k})}_{\w{Z}_1} \right\|_2 + \left\| \underbrace{(\w{T} - \w{I}_{N \times N})[\nabla^2 f_h ({\w{x}_{h,k}})]^{1/2} \w{d}_{h,k}}_{\w{Z}_2} \right\|_2 \\
\end{equation}
Using \cref{proposition f} and since, by assumption, 
$y \hat{\lambda}(\w{x}_{h,k})< 1$ we take,
\begin{equation*}
    [\nabla^2 f_h ({\w{x}_{h,k}})]^{-1/2} \nabla^2 f_h ({\w{x}_{h,k}} +y\w{\hat{d}}_{h,k})
       [\nabla^2 f_h ({\w{x}_{h,k}})]^{-1/2} \preceq  \frac{1}{\left(1 -y \hat{\lambda}(\w{x}_{h,k}) \right)^2} \w{I}_{N \times N}.
\end{equation*}
We are now in position to estimate both norms in \cref{ineq Zs}. For the first we have that,
\begin{align*}
    \left\| \w{Z}_1 \right\| & \leq \left\| \int_0^1 \frac{1}{\left(1 -y \hat{\lambda}(\w{x}_{h,k}) \right)^2} dy  \right\|_2  \left\| \ [\nabla^2 f_h ({\w{x}_{h,k}})]^{1/2} \left(
			   \w{\hat{d}}_{h,k} - \w{d}_{h,k} \right) \right\|_2 \\
			   & = \frac{1}{1 - \hat{\lambda}(\w{x}_{h,k}) } \left\| [\nabla^2 f_h ({\w{x}_{h,k}})]^{1/2} \left(
			   \w{\hat{d}}_{h,k} - \w{d}_{h,k} \right) \right\|_2.
\end{align*}
Using now the result from \cref{lemma norm equality lambda's d's}, 
we obtain,
\begin{equation*}
 \left\| \w{Z_1} \right\|_2 \leq \frac{1}{1 - \hat{\lambda}(\w{x}_{h,k}) } \sqrt{\lambda(\w{x}_{h,k})^2 - 
 \hat{\lambda}(\w{x}_{h,k})^2}.
\end{equation*}
Next, the second norm implies 
\begin{align*}
\left\| \w{Z}_2 \right\| & \leq    \left\| \int_0^1 \left( \frac{1}{\left(1 -y \hat{\lambda}(\w{x}_{h,k}) \right)^2} - 1 \right) dy \right\|_2 
			   \left\| [\nabla^2 f_h ({\w{x}_{h,k}})]^{1/2} \w{d}_{h,k} \right\|_2 \\
			   			    & = \frac{\hat{\lambda}(\w{x}_{h,k})}{1 - \hat{\lambda}(\w{x}_{h,k}) } \lambda(\w{x}_{h,k}).
\end{align*}
Putting this all together, inequality \cref{ineqlambda} becomes,
 \begin{equation*}
      \lambda(\w{x}_{h,k+1}) \leq  \frac{ \sqrt{\lambda(\w{x}_{h,k})^2 - 
 \hat{\lambda}(\w{x}_{h,k})^2}}{\left( 1 - \hat{\lambda}(\w{x}_{h,k}) \right)^2} + \frac{\hat{\lambda}(\w{x}_{h,k})}{\left( 1 - \hat{\lambda}(\w{x}_{h,k}) \right)^2}
   \lambda(\w{x}_{h,k}).
 \end{equation*}
as claimed. \qed
\end{proof}

\begin{proof}[Proof of \cref{lemma norm on d's upper bound}]
We have that,
\begin{align*}
    &\left\| [\nabla^2 f_h ({\w{x}_{h,k}})]^{1/2} \left(\w{\hat{d}}_{h,k} - \w{d}_{h,k} \right) \right\|_2  = \\
    & \left\| [\nabla^2 f_h ({\w{x}_{h,k}})]^{1/2} \left( [\nabla^2 f_h ({\w{x}_{h,k}})]^{-1} - \w{P} \left[\w{R} \nabla^2 f_h ({\w{x}_{h,k}}) \w{P}  \right]^{-1} \w{R} \right) \nabla f_h ({\w{x}_{h,k}}) \right\|_2 \leq  \\
    & \left\|  \w{I} - [\nabla^2 f_h ({\w{x}_{h,k}})]^{1/2} \w{P} \left[\w{R} \nabla^2 f_h ({\w{x}_{h,k}}) \w{P}  \right]^{-1} \w{R} [\nabla^2 f_h ({\w{x}_{h,k}})]^{1/2} \right\|_2 \lambda(\w{x}_{h,k}) =  \\
    & \left\|  \w{I} - \Pi_{[\nabla^2 f_h (\w{x}_{h,k})]^{1/2} \w{P} } \right\|_2 \lambda(\w{x}_{h,k}), \\
\end{align*}
where $\Pi_{[\nabla^2 f_h (\w{x}_{h,k})]^{1/2} \w{P} }$ is the orthogonal projection onto the $\operatorname{span}\{ [\nabla^2 f_h(\w{x}_{h,k})]^{\frac{1}{2}} \w{P}\}$. Since $\w{I} - \Pi_{[\nabla^2 f_h (\w{x}_{h,k})]^{1/2} \w{P} }$ is also an orthogonal projection we conclude that,
\begin{equation*}
     \left\| [\nabla^2 f_h ({\w{x}_{h,k}})]^{1/2} \left(\w{\hat{d}}_{h,k} - \w{d}_{h,k} \right) \right\|_2 \leq \lambda(\w{x}_{h,k}),
\end{equation*}
as claimed. \qed
\end{proof}

\begin{proof}[Proof of \cref{lemma bounds unknown quantity}]
Combining \cref{lemma upper bound lamda_hat} and the fact that
$\hat{\lambda}(\w{x}_{h,k}) > \nu$
we have that $0 < \frac{\hat{\lambda}(\w{x}_{h,k})}{\lambda(\w{x}_{h,k})} \leq 1$ with probability $1-\rho$ whenever $\w{x}_{h,k} \neq \w{x}_h^*$, $k \in \mathbb{N}$. Then, there exists $\mu_k \in \left(0, \frac{\hat{\lambda}(\w{x}_{h,k})}{\lambda(\w{x}_{h,k})} \right] $ and thus $0 < \mu_k \leq 1$. Further, $0 < \mu_k \leq \frac{\hat{\lambda}(\w{x}_{h,k})}{\lambda(\w{x}_{h,k})}$ implies,
\begin{align*}
    1 - \frac{\hat{\lambda}(\w{x}_{h,k})^2}{\lambda(\w{x}_{h,k})^2} \leq & 1 - \mu_k^2 < 1 \\
    \lambda(\w{x}_{h,k})^2 - \hat{\lambda}(\w{x}_{h,k})^2 \leq & (1 - \mu_k^2) \lambda(\w{x}_{h,k})^2 < \lambda(\w{x}_{h,k})^2,
\end{align*}
with probability $1-\rho$, as required. \qed
\end{proof}

\section{Further Numerical Experiments and Details} \label{sec num res}

In this section we provide full details of the experiments presented in \cref{sec experiments} of the main text as well as additional experiments. Specifically, in the first part, we describe three cases of Generalized Linear Models (GLMs) for solving the problem of maximum likelihood estimation that are used in numerical experiments. In the second part, we present the algorithms used in the comparisons against SIGMA and we describe how to efficiently compute the reduced Hessian matrix.

\subsection{Generalized Linear Models} \label{sec glms}

We consider solving the maximum likelihood estimation problem based on the Generalized Linear Models (GLMs) that are widely used in practice for prediction and classification problems. Given a collection of data points $ \{ \w{a}_i, b_i \}_{i = 1}^m $, where $\w{a}_i \in \mathbb{R}^N$, the problem has the form
\begin{equation*}
    \min_{\w{x}_{h} \in \mathbb{R}^N} f_h(\w{x}_h) := \frac{1}{m} \sum_{i=1}^m f_{h, i} (\langle \w{\w{a_i}}, \w{x}_{h} \rangle, b_i),
\end{equation*}
where $f$ is a GLM and we further assume that it is convex and bounded below so that the minimizer exists and is unique. When the user desires to enforce specific structure in the solution, $f$ is called a regularized GLM. Typical instances of regularization for GLMs are $\ell_1$- and/or $\ell_2$-regularization and thus the problem of maximum likelihood estimation is written
\begin{equation*} \label{regularized glm}
    \min_{\w{x}_{h} \in \mathbb{R}^N} f_h(\w{x}_h) + \xi_2 \|\w{x}_h \|_2^2 + \xi_1 g(\w{x}_h),
\end{equation*}
where $\xi_1$ and $\xi_2$ are positive numbers and $g(\w{x}_h) := \sum_{i=1}^N [ (c^2 + x_{h, i}^2)^{\frac{1}{2}} - c ]$, for some $c > 0$, is the pseudo-Hubert function which is a smooth approximation of the $\ell_1$-norm and provides good approximations for small $c$ \cite{fountoulakis2016second}. Below we present three special cases of GLMs. We discuss which of them satisfy \cref{ass self-conc} and comment on their practical implementation.

\textbf{Gaussian linear model:} $f_{h,i} (\langle \w{a_i}, \w{x}_{h} \rangle, b_i) := \frac{1}{2}( \w{a_i}^T \w{x}_{h}  - b_i)^2$, where $b_i \in \mathbb{R}$. This problem corresponds to the standard least-squares method. When $ \{ \w{a}_i \}_{i = 1}^m  $ are linearly independent vectors \cref{ass self-conc} is fulfilled. It holds that,
\begin{equation*}
    \nabla f_h (\w{x}_{h}) = \frac{1}{m} \sum_{i=1}^m \w{a}_i (\w{a}_i^T \w{x}_h - b_i) \ \ \text{\&} \ \ \nabla^2 f_h (\w{x}_{h}) = \frac{1}{m} \sum_{i=1}^m \w{a}_i \w{a}_i^T,
\end{equation*}
and therefore, using the above gradient and Hessian, one shall implement \cref{gama} with backtracking line search. 

\textbf{Poisson model with identity link function:} Initially, the model has the form 
$\tilde{f}_{h, i} (\langle \w{\w{a_i}}, \w{x}_{h} \rangle, b_i) := \w{a}_i^T \w{x}_h - b_i \log(\w{a}_i^T \w{x}_h)$ where, for 
$i = 1, \ldots, m$, $\operatorname{dom} f_{h, i} = \{ \w{x}_h  \in \mathbb{R}^N : \langle \w{a}_i \w{x}_h \rangle > 0 \}$ and we consider 
counted-valued responses, i.e.,  $b_i \in \mathbb{N}_{+}$. As a standard application of the Poisson model consider the low-light imaging 
reconstruction problem in \cite{harmany2011spiral}. Next, note that when $ \{ \w{a}_i \}_{i = 1}^m  $ are linearly independent vectors the function 
$\tilde{f}_h(\w{x}_h) := \frac{1}{m} \sum_{i=1}^m \tilde{f}_{h, i} (\langle \w{\w{a_i}}, \w{x}_{h} \rangle, b_i)$ is strictly convex self-concordant 
with constant $M_{\tilde{f}} = 2 \sqrt{m} \max \{\frac{1}{\sqrt{b_i}}, b_i \in \mathbb{N}_+, i = 1, \ldots, m \}  $ \cite[Theorem 4.1.1]{MR2142598} 
and thus $f_h (\w{x}_h) := \frac{M_{\tilde{f}}^2}{4}  \tilde{f}_h (\w{x}_h)$ satisfies \cref{ass self-conc} (see \cref{sec background}). As a 
result, one shall implement \cref{gama} using the scaled Poisson model which has gradient and Hessian matrix
\begin{equation*}
    \nabla f_h (\w{x}_{h}) = \frac{M_{\tilde{f}}^2}{4} \sum_{i=1}^m  (1 - \frac{b_i}{\w{a}_i^T \w{x}_h}) \w{a}_i \ \ \text{\&} \ \  \nabla^2 f_h (\w{x}_{h}) = \frac{M_{\tilde{f}}^2}{4} \sum_{i=1}^m \frac{b_i}{(\w{a}_i^T \w{x}_h)^2} \w{a}_i \w{a}_i^T,
\end{equation*}
respectively. Since at every step $k$ the update $\w{x}_{h,k} + t_{h,k} \w{d}$ must belong in $\operatorname{dom} f $ we adopt the following step size strategy for finding an appropriate initial value in $t_{h,k}$: By the proved optimal step-size parameter in \cref{lemma411} we start by selecting $t_{h,k} = \frac{1}{1 + \sqrt{ - \langle f_{h,k} (\w{x}_{h,k}), \w{d}  \rangle}}$ and then we increment it by a constant $\zeta > 1$, i.e., $\tilde{t}_{h,k} = \zeta t_{h,k}$. We repeat this procedure as long as $\w{x}_{h,k} + \tilde{t}_{h,k} \w{d} \in \operatorname{dom} f$. Then we initiate the backtracking line search with $ t_{h,k} = \min \{ \tilde{t}_{h,k}, 1 \}$, where $\tilde{t}_{h,k}$ is the largest value that satisfies $\w{x}_{h,k} + \tilde{t}_{h,k} \w{d} \in \operatorname{dom} f$ obtained by the previous procedure.

\textbf{Logistic model:} $ f_{h, i} (\langle \w{\w{a_i}}, \w{x}_{h} \rangle, b_i) := \log (1 + e^{-b_i \w{a}_i^T \w{x}_h})$, for binary responses, i.e., $b_i = \{ 0,  1\}$. Notice that the logistic model is \textit{not} self-concordant for all $\w{x}_h \in \operatorname{dom} f = \mathbb{R}^N$, nevertheless, we wish to illustrate the efficacy of \cref{gama} when \cref{ass self-conc} is not satisfied. The gradient and Hessian matrix of the logistic model are given by,
\begin{equation*}
    \nabla f_h (\w{x}_{h}) = - \frac{1}{m} \sum_{i=1}^m b_i \frac{e^{-b_i \w{a}_i^T \w{x}_h}}{1 + e^{-b_i \w{a}_i^T \w{x}_h}} \w{a}_i \ \ \text{\&} \ \     \nabla^2 f_h (\w{x}_{h}) =  \frac{1}{m} \sum_{i=1}^m b_i^2 \frac{e^{-b_i \w{a}_i^T \w{x}_h}}{(1 + e^{-b_i \w{a}_i^T \w{x}_h})^2} \w{a}_i \w{a}_i^T, 
\end{equation*}
respectively. \cref{gama} for the logistic model will be implemented with back-tracking line search.

We also clarify that the pseudo-Hubert function is not self-concordant whereas, on the hand, the $\ell_2$-norm does satisfy  \cref{ass self-conc} as a quadratic function.
Finally, for all three examples, \cref{gama} can be generated either by the Nystr\"om method or with any $\w{P}$ as long as \cref{assumption P} holds.

\subsection{Experiments}

\begin{table}[t]
    \centering
	\begin{tabular}{ |l|c|c|c|c|c| }
		\hline
		Datasets & Problem & $m$ & $N$ & $\xi_1$ & $\xi_2$  \\
		\hline
		CTslices & Gaussian model & $53,500$ & $385$ & $ 0 $ & $10^{-6}$  \\
		
		CMHS & Gaussian model & $ 2,204 $& $ 19,320 $ & $ 0 $,  $10^{-6}$& $10^{-6}$ \\
		
		Synthetic & Poisson model & $1,000$  & $ 800 $ & $ 0 $, $10^{-3}$ & $10^{-6}$ \\
		
		Gissette & Logistic model & $6,000$  & $5,000$ & $ 10^{-6} $ & $10^{-6}$\\
		
		Leukemia & Logistic model & $ 38 $  & $7,129$ & $0$ & $ 10^{-6} $\\
		
		Real-sim & Logistic model & $ 72,309 $  & $20,958$ & $0$ & $ 10^{-6} $\\
		\hline
	\end{tabular}
    \caption{Problem characteristics and corresponding datasets used in the experiments. All datasets are available from   
    \url{https://www.csie.ntu.edu.tw/~cjlin/libsvmtools/datasets/} and \url{http://archive.ics.uci.edu/ml/index.php}}
    \label{tab: datasets}
\end{table}

For each problem presented in the previous section, we experimented over real and synthetic datasets. To illustrate that the proposed algorithm is suitable for a wide range of problems, we consider regimes in which we have $m > N$ or $m < N$ (or even $m \gg N$ or $m \ll N$), for more details see \cref{tab: datasets}. What follows is a description of the algorithms used in comparisons against conventional SIGMA:
\begin{enumerate}
    \item Gradient Descent (GD) with back-tracking line search.
    \item Stochastic Gradient Descent (SGD) with step size rule $t_k = \frac{t}{ 1 + \gamma k}$, where $\gamma $ is set to $10^{-6}$ and $t$ was tuned manually depending on each problem.
    \item Newton method with backtracking line search.
    \item Sub-sampled Newton method (SubNewton) with back-tracking line search and manually tuned number of samples $|S_m|$  \cite{berahas2020investigation}.
    \item NewSamp with back-tracking line search and manually tuned number of samples $|S_m|$ \cite{NIPS2015_404dcc91}.
\end{enumerate}

For all algorithms, we measure the error (in log scale) using the norm of the gradient $\| \nabla f_h (\w{x}_{h,k}) \| $ and as a stopping criterion we use either if $\| \nabla f_h (\w{x}_{h,k}) \| \leq \epsilon $ or if a maximum CPU time (in seconds) is exceeded. In all experiments, SIGMA is tuned as follows. The Galerkin model is generated by the Nystr\"om method and $\w{P}$ is constructed according to \cref{def P} in each iteration. In the following remark we describe how to efficiently compute the reduced Hessian matrix in \cref{definition Q_H} for the GLMs.

\begin{remark} \label{remark reduced hess}
We clarify that it is expensive to perform the matrix multiplications when computing the reduced Hessian $\w{R} \nabla^2 f_h(\w{x}_{h,k}) \w{P}$ (\cref{definition Q_H}). However, since we make use of the (naive) Nystr\"om method with $\w{P}$ as in \cref{def P}, one need not perform the expensive matrix multiplications associated with the matrix $\w{P}$. Instead, it suffices only to sample (without replacement) the vector $\w{a}_i$. To see this, for example, consider the logistic regression problem.  The reduced Hessian matrix takes the following form
\begin{align*}
 \w{R} \nabla^2 f_h (\w{x}_{h})\w{P} &= \frac{1}{m} \sum_{i=1}^m b_i^2 \frac{e^{-b_i \w{a}_i^T \w{x}_h}}{(1 + e^{-b_i \w{a}_i^T \w{x}_h})^2} \w{R} \w{a}_i  \w{a}^T_i \w{P} \\ 
 & = \frac{1}{m} \sum_{i=1}^m b_i^2 \frac{e^{-b_i \w{a}_i^T \w{x}_h}}{(1 + e^{-b_i \w{a}_i^T \w{x}_h})^2} \w{R} \w{a}_i (\w{R} \w{a}_i)^T .
\end{align*}
Thus, instead of computing the product $\w{R} \w{a}_i  \w{a}^T_i \w{P}$, it suffices to sample $n$ from $N$ entries of the vector $\w{a}_i$ before forming the matrix $\w{a}_i  \w{a}^T_i$. This fact yields much faster iterations, i.e.,  $ \mathcal{O}(m n^2) $ operations are required to form the reduced Hessian matrix.
\end{remark}

Further, we use \textit{no} checking condition on whether to perform coarse or fine steps and thus at each iteration the coarse model is always employed ---the condition $\hat{\lambda}(\w{x}_{h,k}) > \nu $ can be trivially verified in practice for the values of $n$ used in our implementations. The extra results of the following two sections concern the conventional SIGMA, which uses the naive Nystr\"om method. Finally, the results were obtained using a standard desktop computer on a CPU and a Python implementation.

\subsubsection{Linear Regression} \label{subsec linear regr}

\begin{figure}[htbp]
	\begin{subfigure}{.3\textwidth}
		\centering
		\includegraphics[width=1\linewidth]{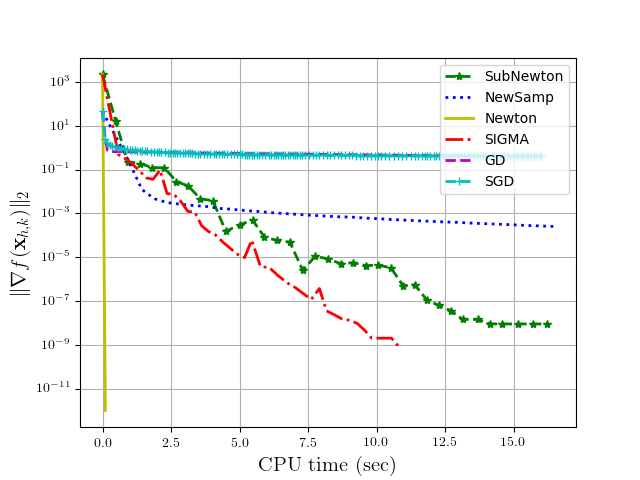}
		\caption{CTslices}
		\label{subfig ctslices}
	\end{subfigure}
	\begin{subfigure}{.3\textwidth}
		\centering
		\includegraphics[width=1\linewidth]{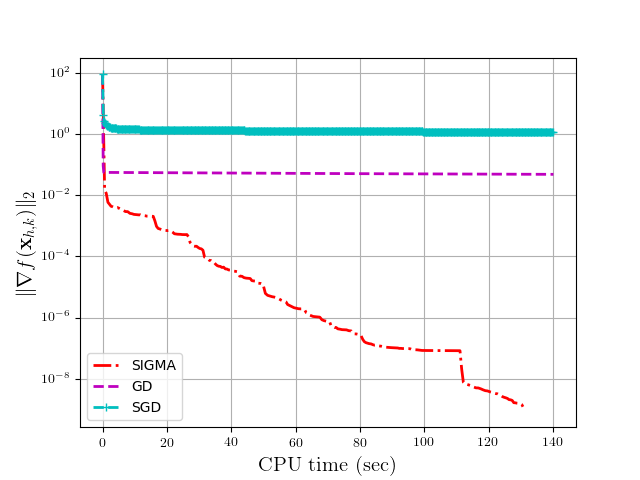}
		\caption{CMHS}
		\label{subfig cmhs}
	\end{subfigure}
	\begin{subfigure}{.3\textwidth}
		\centering
		\includegraphics[width=1\linewidth]{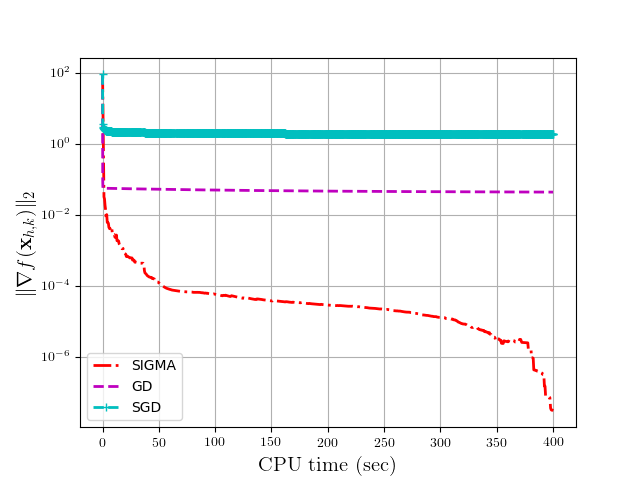}
		\caption{CMHS with $\xi_1 = 10^{-6}$}
		\label{subfig cmhsl1}
	\end{subfigure}
	\caption{Comparison of different algorithms over various datasets for linear regression. Error vs CPU time in seconds.}
	\label{fig gaussian}
\end{figure}

We consider two different regimes to validate the efficiency of SIGMA over the regularized Gaussian model: (i) when $ m \gg N$ and, (ii) when $m \ll N$.

\textbf{Regularized linear regression with $m \gg N$}. The CT slices dataset will serve the purpose of this regime (see  \cref{tab: datasets}). This problem is otherwise called as Ridge regression. The coarse model dimensions for SIGMA is set to $n = 0.7 N$ while for the sub-sampled Newton methods we randomly select $|S_m| = m/2$ samples to form the Hessian matrix at each iteration. The performance between the optimization methods for this example can be found in \cref{subfig ctslices}. As expected, the Newton method converges to the solution just after one iteration. Besides the Newton method, SIGMA and the sub-sampled Newton are the only algorithms that are able to reach very high accuracy. In particular, \cref{subfig ctslices} suggests that both methods enjoy a very fast rate while GD, SGD and NewSamp only achieve a slow linear rate. We emphasize that the sub-sampled Newton method  is particularly well suited for this regime, nevertheless SIGMA offers comparable, if not better, results indicating that it is efficient even for problems in this regime.

\textbf{Regularized linear regression with $m \ll N$}. The real strength of SIGMA emerges when $m \ll N$ with very large $N$. By using the Condition monitoring of hydraulic systems (CMHS) dataset, the problem dimensions grow very large and as a result the second-order methods used in the previous example are not applicable here due to memory limitations. Therefore, for this exmaple,  \cref{subfig cmhs} compares the performance between SIGMA, GD and SGD. Clearly, SIGMA compares favourably to the first-order methods as it is able to reach a very accurate solution. On the other hand, the first-order achieve a very slow convergence rate and are way far from the solution by the time SIGMA converges. In this regime, it is typical one to require enforcing sparsity in the solution. For this reason, in the next experiment we solve the same problem in which, now, in addition to $\ell_2$-norm, the pseudo-Hubert function is activated with $\xi_1 = 10^{-6}$. This form of regularization is what is also called elastic-net. Again, SIGMA outperforms its competitors (see \cref{subfig cmhsl1}). We recall that this instance of linear regression does not satisfy \cref{ass self-conc}. However, it is interesting to observe in \cref{subfig cmhsl1} that, before convergence, SIGMA achieves a sharp super-linear rate that approaches the quadratic rate of the Newton method. We highlight also that in this regime the Hessian matrix has $\operatorname{rank} (\nabla^2 f_h (\w{x}_{h})) = m$. To guarantee that the Hessian matrix is invertible we add a small amount $\xi_2$ in the diagonal of $\nabla^2 f_h (\w{x}_{h})$ and therefore there is a big gap between the $m$ and $m+1$ eigenvalues. As a result, the convergence behavior of SIGMA in \cref{subfig cmhs} and \cref{subfig cmhsl1} verifies our intuition regarding the efficiency of SIGMA in such problem structures, as well as the theoretical results of \cref{sec nystrom}.

\subsubsection{Sub-sampled SIGMA} \label{sec sub-sampled sigma}

\begin{figure}[htbp]
		\begin{subfigure}{.47\textwidth}
		\centering
		\includegraphics[width=1\linewidth]{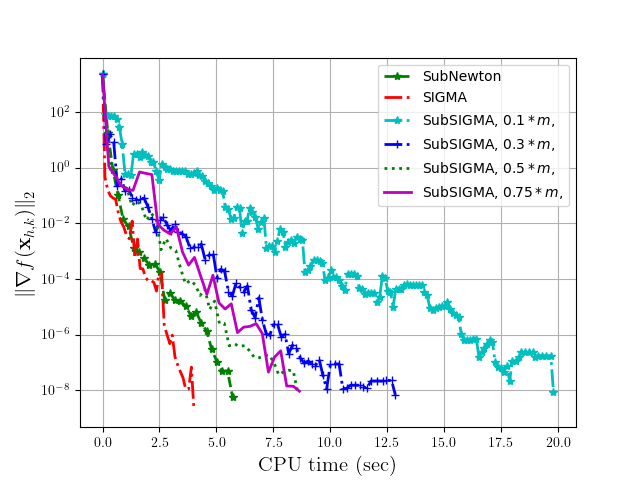}
		\caption{CT slices}
		\label{subfig subCTslices}
	\end{subfigure}
	\begin{subfigure}{.47\textwidth}
		\centering
		\includegraphics[width=1\linewidth]{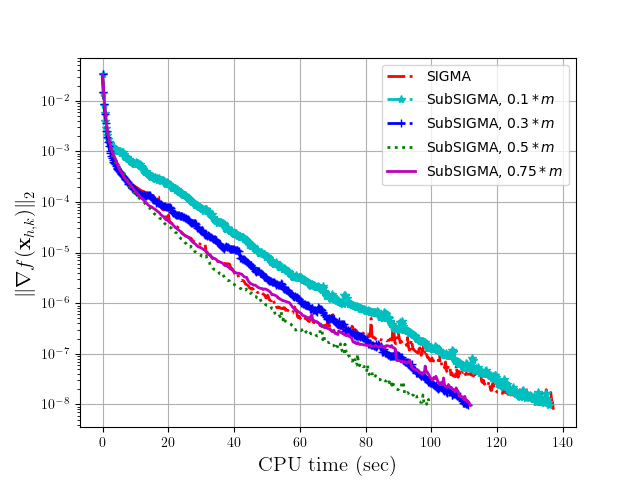}
		\caption{Real-sim}
		\label{subfig sub-realisim}
	\end{subfigure}
	\caption{Performance of sub-sampled SIGMA on the Gaussian and logistic models.}
	\label{fig sub sigma}
\end{figure}

In the last set of experiments we revisit the CTslices and Real-sim datasets for minimizing the Gaussian and Logistic models, respectively, but now we implement conventional SIGMA with sub-sampling. Although the analysis of SIGMA with sub-sampling is beyond the scope of this paper, in this section, we wish to identify particular examples in order to demonstrate, through numerical experiments, an improved convergence rate when solving the coarse model that is generated by the naive Nystr\"om method with sub-sampling. To compute the reduced Hessian matrix (eq. \cref{definition Q_H}) with sub-sampling we follow the procedure described in \cref{remark reduced hess} and, in addition, we sample (uniformly without replacement) $|S_m|$ from $m$ data points. For instance, the reduced Hessian matrix of the logistic model takes the following form
\begin{align*}
 \w{R} \nabla^2 f_h (\w{x}_{h})\w{P} 
 & = \frac{1}{|S_m|} \sum_{i \in S_m} b_i^2 \frac{e^{-b_i \w{a}_i^T \w{x}_h}}{(1 + e^{-b_i \w{a}_i^T \w{x}_h})^2} \w{R} \w{a}_i (\w{R} \w{a}_i)^T .
\end{align*}
As a result, the cost of forming the reduced Hessian matrix with sub-sampling will be $\mathcal{O} (|S_m| n^2) $
which yields faster iterations compared to the standard SIGMA which requires $\mathcal{O} (m n^2) $ for the same purpose. However, as we illustrate in \cref{fig sub sigma}, this, in general, does not necessarily mean faster convergence.

In \cref{fig sub sigma} we provide comparisons between  SIGMA, sub-sampled SIGMA and the sub-sampled Newton method. As we revisit the examples from sections \ref{subsec linear regr} and \ref{subsec log regr} (\cref{subfig subCTslices} and \ref{subfig sub-realisim}, respectively), the sub-sampled Newton method and SIGMA are tuned as before. In addition, in order to capture the sub-sampling effect on SIGMA, we consider different values for the sub-sampling parameter: $|S_m| = \{0.1m, 0.3m, 0.5m, 0.75m \}$. \cref{subfig subCTslices} shows that SIGMA and the sub-sampled Newton method outperform all instances of sub-sampled SIGMA when solving the Gaussian model for the CTslices dataset. Note that in this example, the problem dimensions $N$ is only $385$ and thus it is natural to expect that sampling both $m$ and $N$ will result in a slower convergence since much of the second-order information is lost. For this reason, in the second experiment we consider the Real-sim dataset over the logistic model for which both $m$ and $N$ are quite large. In this case, \cref{subfig sub-realisim} shows that great improvements are achieved for all instances of the sub-sampled SIGMA particularly when aiming for very accurate solutions. Specifically, sub-sampled SIGMA with $|S_m| = 0.5m$ significantly outperforms the standard SIGMA while, in the worst case, the sub-sampled SIGMA with $|S_m| = 0.1m$ offers comparable results to the standard SIGMA. To this end, as illustrated by \cref{fig sub sigma}, SIGMA with sub-sampling will be more suitable to problems with very large $m$ and $N$ since in this case the computational bottleneck appears in both when evaluating the Hessian matrix and solving the corresponding system of linear equations. Such problems lie at the core of large-scale optimization and therefore methods that exhibit the advantages of the sub-sampled SIGMA can potentially offer a powerful tool for solving complex problems that arise in modern machine learning applications. 

\bibliographystyle{plain}
\bibliography{references}
%
%
%
%
%
%
%
%
%

\end{document}